\documentclass[11pt]{article}
\usepackage[english]{babel}
\usepackage{amsmath, amsfonts, amsthm, amssymb, amscd,enumerate}
\newtheorem{theorem}{Theorem}[section]
\newtheorem{prop}{Proposition}[section]
\newtheorem{lemma}{Lemma}[section]
\newtheorem{Cor}{Corollary}[section]
\newtheorem{nb}{Remark}[section]

\newtheorem{Def}{Definition}[section]
\usepackage{pdfpages}

\begin{document}
\title{On the Clifford short-time Fourier transform and its properties }
\date{}
\author{Antonino De Martino}
\maketitle
\begin{abstract}
\noindent
In this paper we investigate how the short-time Fourier transform can be extended in a Clifford setting. We prove some of the main properties of the Clifford short-time Fourier transform such as the orthogonality relation, the reconstruction property and the reproducing kernel formula. Moreover, we show the effects of modulating and translating the signal and the window function, respectively. Finally, we demonstrate the Lieb's uncertainty principle for the Clifford short-time Fourier transform.
\end{abstract}


\section{Introduction}
Recently there has been an increasing interest in the generalizations of the Fourier transform to the quaternionic and to the Clifford algebras settings. These transforms can help in the study of vector-valued signals and images. Moreover, in \cite{CK} the authors
explained that some hypercomplex signals are useful tools for extracting intrinsically 1D-features from images.
\\ De Bie H. explains in \cite{D} that nowadays the emphasis of the research is on three different methods: the eigenfunction approach, the generalized roots of $-1$ approach and the spin group approach.

In the first one the Clifford-Fourier transform is defined as the following integral transform
\begin{equation}
\label{first}
\mathcal{F}_{\pm} f(y):= (2 \pi)^{-\frac{d}{2}} \int_{\mathbb{R}^d} K_{\pm}(x,y) f(x) \, dx,
\end{equation}
where the kernel $ K_{\pm}(x,y)$ is given by an explicit expression. For this kind of Clifford-Fourier transform many generalizations were found, see \cite{DS, DSS, BDS, DOV,CDL} and some important properties such as the uncertainty principle and Riemann-Lebesgue lemma were proved \cite{EJ, EJ1}.

In the second approach the definition of Fourier transform is given in the following way.
\begin{Def}
Denote by $ \mathcal{I}_d$ the set $ \{i \in \mathbb{R}_{d}| \, i^2=-1\}$ of geometric square roots of minus one, where $ \mathbb{R}_d$ denotes the real Clifford algebra. Let $F_1:= \{i_1,..., i_\mu \}$, $F_2= \{i_{\mu +1},..., i_d\}$ be two ordered finite sets of such square roots, $i_k \in \mathcal{I}_d$, for all $ k=1,...,d$. The geometric Fourier transform $ \mathcal{F}_{F_1, F_2}$ of a  function $f : \mathbb{R}^d \to \mathbb{R}_{d}$ takes the form
$$ \mathcal{F}_{F_1, F_2}(f)(u):= (2 \pi)^{- \frac{d}{2}} \int_{\mathbb{R}^d} \biggl( \prod_{k=1}^{\mu} e^{- i_k x_k u_k} \biggl) f(x) \biggl( \prod_{k=\mu +1}^{d} e^{- i_k x_k u_k} \biggl) \, dx.$$
\end{Def}
Finally, in the third approach the Fourier transform is defined as follows
$$ \hat{f}(\phi):= \int_{\mathbb{R}^2} \phi(x,y) f(x,y) \phi(-x, -y) \, dx \, dy,$$
where $ \phi$ is a group of morphism.

The aim of this paper is to generalize the short-time Fourier transform
$$ \mathcal{V}_{g}f(x, \omega)= \int_{\mathbb{R}^d} f(t) \overline{g(t-x)} e^{-2 \pi i t \cdot \omega} \, dt, \qquad x, \omega \in \mathbb{R}^d,$$
on the Clifford setting using the first approach. This kind of integral transform can be applied successfully to optical flow and also it should be possible to apply it to image features and image fusion.

A quaternionic short-time Fourier transform in two dimensions is defined in \cite{BA}, where the authors use the second approach. Following the same perspective, some results on time-frequency analysis are proved in the papers \cite{CHO, FKC}. Moreover, a quaternionic short-time Fourier transfrom in dimension one with a gaussian window is defined in \cite{DK}.
We note that there are other possible approaches to Fourier transform in the hypercomplex setting, but they are not related to our approach.

In order to present our results we collect  in section 2 few basic definitions and preliminaries about Clifford algebras and analysis. Section 3 contains the main properties of the Clifford-Fourier transform. In section 4 we give a brief overview of the generalized translation operator \cite{DX}:
$$ \tau_y f(x)= (2 \pi)^{- \frac{d}{2}} \int_{\mathbb{R}^d} K_{-}( \xi,x) K_{-}(y, \xi) \mathcal{F}_{-}f(\xi) \, d \xi, \qquad x ,y \in \mathbb{R}^d.$$
By fixing the well-known property of the Fourier transform of modulation we obtain the definition of  generalized modulation operator:
$$ M_{y} f(x)=K_{-} (x,y)f(x), \qquad x,y \in \mathbb{R}^d.$$
The main difference from the classical case is the fact that there is not a commutative relation  with the generalised translation operator.

In section 5, using formula \eqref{first} and the properties of the generalized translation we derive the expression of the Clifford short-time Fourier transform:
$$ \mathcal{V}_{g}f(x, \omega)= (2 \pi)^{- \frac{d}{2}} \int_{\mathbb{R}^{d}} K_{-}(t, \omega)g(t-x)f(t) \, dt,$$
for $d$ even, $d>2$ and $g$ a radial function called window function.

In section 6 using the generalized modulation and translation we write the Clifford short-time Fourier transform  in different ways, such as
$$ \mathcal{V}_{g}f(x, \omega)= (2 \pi)^{- \frac{d}{2}} \int_{\mathbb{R}^d} \overline{M_{ \omega} \tau_{x} g(t)} f(t) \, dt.$$

In section 7 we derive some sharp inequalities of the Clifford short-time Fourier transform where we modulate and translate the signal and only translate the window function. Surprisingly the estimates depend on the Clifford-Fourier transform of $f$ and the convolution between the window function $g$ and a polynomial. 
$$|\mathcal{V}_{\tau_{\theta} g} \tau_{\mu} M_{\eta} f (x, \omega)| \leq c (1+|\underline{\omega}|)^{\lambda} (1+|\underline{\mu}|)^{\lambda} (1+|\underline{\theta}|)^{2 \lambda} \left( (1+|.|)^{2 \lambda} * |g| \right)(x)\int_{\mathbb{R}^d} (1+ |.|)^{2 \lambda} |\tau_{\eta} \mathcal{F_{-}}f(.)| d \cdot $$

$$
|\mathcal{V}_{\tau_{\theta}  g} M_{\eta} \tau_{\mu}  f (x, \omega)| \leq \!  c (1+|\underline{\omega}|)^{\lambda} (1+|\underline{\eta}|)^{\lambda} (1+|\underline{\mu}|)^{\lambda}  (1+|\underline{\theta}|)^{3 \lambda} \left( \!  (1+|.|)^{3 \lambda}\!  * \!|g| \right)\!(x) \! \! \int_{\mathbb{R}^d} \! \! (1+ |.|)^{2 \lambda} |\mathcal{F_{-}}f(.)| d \cdot 
$$
where $.$ is a fixed variable and $c$ is a positive constant.
Moreover, we observed why it is not possible to have an estimate when we modulate also the window function.
Section 8 contains some important properties of the Clifford short-time Fourier transform. At the beginning, we prove the orthogonality relation
$$ \int_{\mathbb{R}^{2d}} \overline{\mathcal{V}_{g_1}f_1(x, \omega)} \mathcal{V}_{g_2}f_2(x, \omega) \, d \omega \, dx = \langle f_1, f_2 \rangle \biggl( \int_{\mathbb{R}^d} g_1(z) g_2(z) \, dz \biggl).$$
Using this we get the reconstruction formula
$$ f(y)= \frac{1}{\int_{\mathbb{R}^d} g^2(z) \, dz} \int_{\mathbb{R}^{2d}}M_{\omega} \tau_x g(y) \mathcal{V}_{g}f(x, \omega) \, d\omega \, dx.$$
Finally, from the reconstruction formula we write the Clifford short-time Fourier transform as
$$ \mathcal{V}_g f(x', \omega')=\int_{\mathbb{R}^{2d}}  \mathbb{K}_g(\omega,x ; \omega', x') \mathcal{V}_g f(x, \omega) \, d \omega \, dx .$$
where $ \mathbb{K}_g(\omega,x ; \omega', x')$ is a reproducing kernel defined in the following way

$$ \mathbb{K}_g(\omega,x ; \omega', x'):= \frac{1}{(2 \pi)^{\frac{d}{2}}\int_{\mathbb{R}^d} g^2(z) \, dz} \int_{\mathbb{R}^d} K_{-}(t, \omega') g(t-x') \overline{K_{-}(t, \omega) g(t-x)} \, dt.$$
In the last section we prove that if we assume $U$ to be an open set of $ \mathbb{R}^{2d}$  and
$$ \int \int_{U}(1+| \underline{x}|)^{-2 \lambda} (1+| \underline{\omega}|)^{-2 \lambda} |\mathcal{V}_g f(x, \omega)|^2 \, d x \, d \omega \geq 1- \varepsilon.$$
then  $|U| \geq (1- \varepsilon) \frac{1}{c}$.
This can be regarded as a version of the Lieb's uncertainty principle for the Clifford short-time Fourier transform.

\section{Preliminaries and notations}
Let us denote  by $ \mathbb{R}_d$ the real Clifford algebra generated by $d$ imaginary units $e_1,...,e_d$. The multiplication in this algebra is determined by
$$ e_j e_k+e_ke_j=-2 \delta_{ij}.$$
An element in this algebra can be written as
$$ x= \sum_{A} e_A x_A \quad x_A \in \mathbb{R},$$
where the basis elements $e_A=e_{j_1}...e_{j_k}$ are defined for every subset $ A= \{j_1,...,j_k\}$ of $ \{1,...,d\}$ with $j_1<...<j_k.$ For empty set, one puts $e_{\emptyset}=1$, the later being the identity element.

Observe that the dimension of $ \mathbb{R}_d$ as a real linear space is $2^d$. Moreover, we denote by $ \bar{}$ the Clifford conjugate which is defined as
$$ \overline{\lambda}=\lambda^c, \, \, \lambda \in \mathbb{C}; \qquad \overline{e_i}=-e_i, \, i=1,...,d; \quad \overline{ab}= \bar{b} \bar{a}, \quad a,b \in \mathbb{R}_d,$$
where $^c$ is the complex conjugate.

In the Clifford algebra $\mathbb{R}_d$, we can identify the so-called $1$-vectors, namely, the linear combination with real coefficients of the elements $e_j$, $j=1,...,d$, with the vectors in the Euclidean space $\mathbb{R}^d$. The correspondence is given by the map $(x_1,...,x_d) \mapsto \underline{x}=x_1e_1+...+x_de_d$. The norm of $1$-vector $\underline{x}$ is defined as
\begin{equation}
| \underline{x}|= \sqrt{x_1^2+...+x_d^2}
\end{equation}
 and clearly $\underline{x}^2=-|\underline{x}|^2$. 
\\The product of two $1$-vectors $x= \sum_{j=0}^{d}x_je_j$ and $y= \sum_{j=0}^{d}y_je_j$ splits into a scalar part and $2$-vector part
$$ \underline{x} \underline{y}=- \langle \underline{x}, \underline{y} \rangle+ \underline{x} \wedge \underline{y},$$
with
$$ \langle \underline{x}, \underline{y} \rangle = \sum_{j=1}^d x_jy_j=- \frac{1}{2}(\underline{x}\underline{y}+\underline{y} \underline{x}),$$
and
$$ \underline{x} \wedge \underline{y} = \sum_{j<k} e_je_k(x_jy_k-x_ky_j)= \frac{1}{2}(\underline{x} \underline{y}-\underline{y} \underline{x}).$$
For any $\underline{x},\underline{y} \in \mathbb{R}_d$, we have $| \underline{x} \underline{y} | \leq 2^d |\underline{x}||\underline{y}|$ and $|\underline{x}+\underline{y}| \leq |\underline{x}|+|\underline{y}|$.

In the sequel, we consider functions defined on $ \mathbb{R}^d$ and taking values in the real Clifford algebra $ \mathbb{R}_d$. Such functions can be expressed as
$$ f(x)= \sum_{A}e_{A}f_A(x),$$
where $e_A=e_{i_1}e_{i_2}...e_{i_k}$, $1 \leq i_1< i_2<...<i_k \leq d$ and $f_A$ are real valued functions. We define the modulus $|f|$ of any function $f$ which takes values in $\mathbb{R}_d$ as
\begin{equation}
|f|:= \sqrt{\sum_{A}|f_A|^2}.
\end{equation}
Now, we indicate the Dirac operator in $ \mathbb{R}^d$ as
$$ \partial_{\underline{x}}= \sum_{i=1}^d e_i \partial_{x_i}.$$
The functions in the kernel of this operator are known as monogenic functions \cite{BDS1,CSSS,DSS1, GM}.
\\We denote by $ \mathcal{P}$ the space of polynomials taking values in $ \mathbb{R}_d$, i.e
$$ \mathcal{P}:= \mathbb{R}[x_1,...,x_d] \otimes \mathbb{R}_d.$$
The set of homogeneous polynomials of degree $k$ is then denoted by $ \mathcal{P}_k$, while the space $ \mathcal{M}_k:= \hbox{ker} \partial_{\underline{x}} \cap \mathcal{P}_k$ is called the space of spherical monogenics of degree $k$.

We denote $L^{p}(\mathbb{R}^d) \otimes \mathbb{R}_d$ as the module of all Clifford-valued functions $f: \mathbb{R}^d \to \mathbb{R}_d$ with finite norm
$$ \| f\|_{p}= \begin{cases}
\biggl( \int_{\mathbb{R}^d}|f(x)|^p dx \biggl)^{\frac{1}{p}} , \qquad 1 \leq p < \infty,\\
\hbox{ess} \sup_{x \in \mathbb{R}^d} |f(x)|, \quad p= \infty,
\end{cases}
$$
where $dx=dx_1...dx_d$ represents the usual Lebesgue measure in $\mathbb{R}^d$. We consider the following inner product for two functions $f,g: \mathbb{R}^d \to \mathbb{R}_d$
$$ \langle f,g \rangle= \int_{\mathbb{R}^d} \overline{f(x)} g(x) \, dx.$$

Let us denote the Schwartz space as $\mathcal{S}(\mathbb{R}^d)$. In the rest of the paper we will consider functions in this space which are radial, so real valued functions, or which takes value in the Clifford algebra $ \mathbb{R}_d$. In this last case we denote the Schwartz space as $ \mathcal{S}( \mathbb{R}^{d}) \otimes \mathbb{R}_d$. For this kind of space it is possible to introduce a basis $ \{ \psi_{j,k,l} \}$ (see \cite{S}), which is defined by
\begin{eqnarray}
\label{new1}
\psi_{2j,k,l}(x) &:=& L^{\frac{d}{2}+k-1}_{j}(|\underline{x}|^2) M_{k}^{(l)} e^{- \frac{|\underline{x}|^2}{2}},\\
\label{new2}
\psi_{2j+1,k,l}(x)&:=&L^{\frac{d}{2}+k}_{j}(|\underline{x}|^2) \underline{x} M_{k}^{(l)} e^{- \frac{|\underline{x}|^2}{2}}, 
\end{eqnarray}
where $j,k \in \mathbb{N}$, $ \{ M_{k}^{(l)} \in \mathcal{M}_k: l=1,..., \hbox{dim} \mathcal{M}_k \}$ is a basis for $ \mathcal{M}_k$, and $L^{\alpha}_j$ are the Laguerre polynomials.

In the sequel we define the commutator of the Clifford operators $A$, $B$ in the following way
\begin{equation}
\label{comm1}
[A,B]=AB-BA.
\end{equation}

\section{Clifford-Fourier Transform}
In this section, we recall the definition and some properties of the Clifford-Fourier transform introduced by De Bie H. and Xu Y. in \cite{DX}.
\begin{Def}[Clifford-Fourier Transform]
On the Schwartz class of functions $ \mathcal{S}(\mathbb{R}^d) \otimes \mathbb{R}_d$, we define the Clifford-Fourier transform as
\begin{equation}
\label{first1}
\mathcal{F}_{\pm} f(y):= (2 \pi)^{-\frac{d}{2}} \int_{\mathbb{R}^d} K_{\pm}(x,y) f(x) \, dx,
\end{equation}
and their inverses as
$$ \mathcal{F}^{-1}_{\pm}f(y):=(2 \pi)^{-\frac{d}{2}} \int_{\mathbb{R}^d} \widetilde{K_{\pm}}(x,y) f(x) \, dx, $$
where
$$ K_{\pm}(x,y):= e^{\mp i \frac{\pi}{2} \Gamma_{\underline{y}}} e^{-i \langle \underline{x}, \underline{y} \rangle},$$
and
$$ \widetilde{K_{\pm}}(x,y):= e^{\pm i \frac{\pi}{2} \Gamma_{\underline{y}}} e^{i \langle \underline{x}, \underline{y} \rangle},$$
$\Gamma_{\underline{y}}:= \frac{\partial_y y-y \partial_y}{2}+ \frac{d}{2}=- \sum_{j<k} e_j e_k (x_j \partial_{x_k}-x_k \partial_{x_j})$ is the Gamma operator.
\end{Def}
In the paper \cite{DX} the authors write an explicit formula of the kernel by using the Gegenbauer polynomials $C^{\lambda}_{k}(\omega)$ and the Bessel functions $ J_{\alpha}(t)$:
$$ K_{-}(x,y)=A_\lambda+B_\lambda+( \underline{x}  \wedge \underline{y})C_\lambda,$$
with
$$ A_{\lambda}=2^{\lambda -1} \Gamma(\lambda +1) \sum_{k=0}^{\infty} (i^d+(-1)^k)(|\underline{x}| |\underline{y}|)^{- \lambda} J_{k+ \lambda} (|\underline{x}| |\underline{y}|) C_{k}^{\lambda}(\langle \xi, \eta \rangle),$$

$$ B_{\lambda}=-2^{\lambda -1} \Gamma(\lambda) \sum_{k=0}^{\infty} (k+ \lambda) (i^d-(-1)^k)(|\underline{x}| |\underline{y}|)^{- \lambda} J_{k+ \lambda} (|\underline{x}| |\underline{y}|) C_{k}^{\lambda}(\langle \xi, \eta \rangle),$$

$$ C_{\lambda}=-(2 \lambda)2^{\lambda-1} \Gamma(\lambda) \sum_{k=0}^{\infty} (i^d+(-1)^k)(|\underline{x}| |\underline{y}|)^{- \lambda-1} J_{k+ \lambda} (|\underline{x}| |\underline{y}|) C_{k-1}^{\lambda+1}(\langle \xi, \eta \rangle),$$
where $ \xi= \frac{\underline{x}}{|\underline{x}|}$, $ \eta= \frac{\underline{y}}{|\underline{y}|}$ and $ \lambda= \frac{d-2}{2}$.

It is important to observe that the kernel is not symmetric, in the sense that $K_{-}(x,y) \neq K_{-}(y,x)$. Hence, we adopt the convention that we always integrate over the first variable in the kernels.
Throughout the paper, we only focus on the kernel $ K_{-}=e^{i \frac{\pi}{2} \Gamma_{\underline{y}}} e^{-i \langle \underline{x}, \underline{y} \rangle}$, from the following proposition (see \cite[Prop. 3.4]{DX}) it is possible to recover $K_{+}(x,y)=e^{- i \frac{\pi}{2} \Gamma_{\underline{y}}} e^{-i \langle \underline{x}, \underline{y} \rangle}$.
\begin{prop}
For $x, y \in \mathbb{R}^{d}$,
$$ K_{+}(x,y)= (K_{-}(x,-y))^c,$$
$$ \widetilde{K_{\pm}}(x,y)= {(K_{\pm}(x,y))}^c,$$
where $^c$ stands for the complex conjugate.
\end{prop}
It is important to remark that in the case $d$ is even $A_{\lambda}$, $B_{\lambda}$, $C_{\lambda}$ are real valued so the complex conjugate in $ K_{-}(x,y)$ can be omitted.

From the explicit expression of the kernel it is possible to derive some easy properties \cite[Prop. 5.1]{DX}.
\begin{prop}
\label{linker}
Let $d=2$. Then the kernel of Clifford-Fourier transform satisfies
$$ K_{-}(x,z)K_{-}(y,z)=K_{-}(x+y,z).$$
If the dimension $d$ is even and $d>2$ then
$$ K_{-}(x,z)K_{-}(y,z)\neq K_{-}(x+y,z).$$
\end{prop}

If we consider the Clifford conjugate and even dimension (see \cite[Prop. 3.5]{LLF}) we have

\begin{equation}\label{inv}
K_{-}(y,x)= \overline{K_{-}(x,y)}.
\end{equation}

Furthermore, the following trivial formula holds

\begin{equation}\label{chan}
K_{-}(-x,y)=K_{-}(x,-y).
\end{equation}

Moreover, in \cite{DX} for even dimensions the kernel is rewritten as a finite sum of Bessel functions (see \cite[Thm. 4.3]{DX})
$$ K_{-}(x,y)=(-1)^{\lambda +1} \biggl( \frac{\pi}{2} \biggl)^{\frac{1}{2}} \bigl(A^{*}_{\lambda}(s,t)+B^{*}_{\lambda}(s,t)+ (\underline{x} \wedge \underline{y}) C^{*}_{\lambda}(s,t) \bigl),$$
where $ s= \langle \underline{x}, \underline{y} \rangle $ and $t=|\underline{x} \wedge \underline{y}|= \sqrt{|\underline{x}|^2 |\underline{y}|^2-s^2}$ and

$$ A^{*}_{\lambda}(s,t)= \sum_{l=0}^{\lfloor \frac{\lambda+1}{2}- \frac{3}{4} \rfloor}s^{\lambda-1-2 l} \frac{1}{2^{l} l!} \frac{\Gamma(\lambda+1)}{\Gamma(\lambda-2l)} \tilde{J}_{(2 \lambda -2l-1)/2}(t),$$

$$ B^{*}_{\lambda}(s,t)= - \sum_{l=0}^{\lfloor \frac{\lambda+1}{2}- \frac{1}{2} \rfloor}s^{\lambda-2l} \frac{1}{2^{l} l!} \frac{\Gamma(\lambda+1)}{\Gamma(\lambda-2l+1)} \tilde{J}_{(2 \lambda -2l-1)/2}(t),$$

$$ C^{*}_{\lambda}(s,t)= - \sum_{l=0}^{\lfloor \frac{\lambda+1}{2}- \frac{1}{2} \rfloor}s^{\lambda-2l} \frac{1}{2^{l} l!} \frac{\Gamma(\lambda+1)}{\Gamma(\lambda-2l+1)} \tilde{J}_{(2 \lambda -2l+1)/2}(t)$$
\noindent
with $ \tilde{J}_{\alpha}(t)=t^{- \alpha} J_{\alpha}(t).$

The above formula helps De Bie H. and Xu Y. to prove a very important estimate of the kernel, \cite[Thm. 5.3]{DX}.
\begin{lemma}
\label{Ker1}
Let d be even. For $x,y \in \mathbb{R}^d$, one has
\begin{equation}\label{Ker}
|K_{-}(x,y)| \leq c(1+|\underline{x}|)^{\lambda}(1+|\underline{y}|)^{\lambda}.
\end{equation}
\end{lemma}
Some problems for the Clifford-Fourier Transform arise when we consider odd dimensions. Indeed in this case it is not known if it is possible to write $K_{-}(x,y)$ as sums of Bessel functions. Moreover, it seems not possible to obtain an upper bound of the kernel as in \eqref{Ker}. So in the rest of the paper we focus only on the even dimensions more than two.

Let us define the following space of functions
\begin{equation}\label{Sp1}
B(\mathbb{R}^d):= \biggl \{ f \in L^1(\mathbb{R}^d): \| f \|_{B}:= \int_{\mathbb{R}^d}(1+|\underline{y}|)^ \lambda |f(y)| \, dy < \infty \biggl \}.
\end{equation}
Due to the boundedness of the kernel we have the following important theorem (see \cite[Thm 6.1]{DX}).

\begin{theorem}
\label{WD}
Let $d$ be an even integer. The Clifford-Fourier transform is well defined on $B(\mathbb{R}^d)\otimes \mathbb{R}_d$. In particular, for $f \in B(\mathbb{R}^d)\otimes \mathbb{R}_d$, $ \mathcal{F}_{\pm}f$ is a continuous function.
\end{theorem}

Now, we list some important properties of the Clifford-Fourier transform.

\begin{prop}\cite[Prop. 3.6]{LLF}[Plancherel theorem]
\label{Plan}
If $f,g \in \mathcal{S}(\mathbb{R}^d) \otimes \mathbb{R}_d$ then
\begin{equation}
\int_{\mathbb{R}^d} \overline{f(y)} g(y) \, dy= \int_{\mathbb{R}^d} \overline{\mathcal{F}_{-}(f)(x)} \mathcal{F}_{-}(g)(x) \, dx.
\end{equation}
\end{prop}

\begin{prop}\cite[Prop. 3.7]{LLF}[Parseval's identity]
\label{Par}
If $f \in \mathcal{S}(\mathbb{R}^d) \otimes \mathbb{R}_d$, then
\begin{equation}
\| f \|_{2}= \| \mathcal{F}_{-}(f) \|_{2}.
\end{equation}
\end{prop}

\begin{theorem}\cite[Thm. 6.6]{DX}
\label{acb}
For the basis $ \{ \psi_{j,k,l} \}$ of $ \mathcal{S}(\mathbb{R}^d) \otimes \mathbb{R}_{d}$ (see \eqref{new1} and \eqref{new2}), one has
$$ \mathcal{F}_{\pm}(\psi_{2j,k,l})=(-1)^{j+k} (\mp 1)^{k} \psi_{2j,k,l},$$
$$ \mathcal{F}_{\pm}(\psi_{2j+1,k,l})= i^{d}(-1)^{j+1} (\mp 1)^{k+d-1} \psi_{2j+1,k,l}.$$
When we restrict to the basis $ \{ \psi_{j,k,l} \}$ we have
\begin{equation}\label{appide}
\mathcal{F}^{-1}_{\pm} \mathcal{F}_{\pm}=Id.
\end{equation}
Moreover, when $d$ is even, \eqref{appide} holds for all $ f \in \mathcal{S}(\mathbb{R}^d) \otimes \mathbb{R}_d$.
\end{theorem}

Now we prove a property, which was proved in a more general setting in \cite[Thm. 6.3]{DSS}.
\begin{theorem}
\label{mine}
When $d$ is even for the basis $ \{ \psi_{j,k,l} \}$ of $ \mathcal{S}(\mathbb{R}^d) \otimes \mathbb{R}_{d}$, one has
\begin{equation}\label{double1}
\mathcal{F}_{\pm} \bigl( \mathcal{F}_{\pm}(\psi_{j,k,l}) \bigl)=\psi_{j,k,l}.
\end{equation}
Moreover, the formula \eqref{double1} holds for all $ f \in \mathcal{S}(\mathbb{R}^d) \otimes \mathbb{R}_{d}$.
\end{theorem}
\begin{proof}
We distinguish two cases depending on $j$.
\newline
If $j$ is even we apply two times Theorem \ref{acb} and we get
$$  \mathcal{F}_{\pm} \bigl( \mathcal{F}_{\pm}(\psi_{j,k,l}) \bigl)=(-1)^{j+k}(-1)^{j+k} (\mp 1)^{k} (\mp 1)^{k} \psi_{2j,k,l}= \psi_{2j,k,l}.$$
If $j$ is odd, as before, we apply two times Theorem \ref{acb} and we obtain
\begin{eqnarray*}
\mathcal{F}_{\pm} \bigl( \mathcal{F}_{\pm}(\psi_{j,k,l}) \bigl) &=& i^d (-1)^{j+1}(\mp 1)^{k+d-1} i^d (-1)^{j+1} (\mp 1)^{k+d-1} \psi_{2j+1,k,l}\\
&=& (i^2)^{d} \psi_{2j+1,k,l}=(-1)^d \psi_{2j+1,k,l}= \psi_{2j+1,k,l}.
\end{eqnarray*}
This proves formula \eqref{double1}.

Finally, since $ \mathcal{F}_{\pm}$ is continuous on $ \mathcal{S}(\mathbb{R}^d) \otimes  \mathbb{R}_d$ \cite[Thm. 6.3]{DX} and $ \{ \psi_{j,k,l} \}$ is a dense subset of $ \mathcal{S}(\mathbb{R}^d) \otimes  \mathbb{R}_d$ we obtain that for all $ f \in \mathcal{S}(\mathbb{R}^d) \otimes  \mathbb{R}_d$

\begin{equation} \label{twof}
\mathcal{F}_{\pm} \bigl( \mathcal{F}_{\pm}(f) \bigl)(x)=f(x).
\end{equation}
\end{proof}
\begin{nb}
Due to the density of $ \mathcal{S}(\mathbb{R}^d) \otimes \mathbb{R}_d$ in $L^2(\mathbb{R}^d) \otimes \mathbb{R}_d$ the Plancherel's theorem, Parseval's identity, and the equations \eqref{appide} and \eqref{twof} can be extended to functions in $ L^{2}(\mathbb{R}^d) \otimes \mathbb{R}_d$.
\end{nb}

\section{Generalized Translation and Modulation}
 In this section we introduce the following operators: the generalized translation and the generalized modulation, which are fundamental tools for developing a time-frequency analysis in a Clifford setting. The generalized translation was introduced for the first time in the paper \cite{DX}, where the authors "fixed" the well-known property of the translation of the classical Fourier transform and derived the following definition.
 \begin{Def}
 \label{tra4}
 Let $ f \in \mathcal{S}(\mathbb{R}^d) \otimes \mathbb{R}_d$. For $ y \in \mathbb{R}^d$ the generalized translation operator $ f \mapsto \tau_y f$ is defined by
 $$ \mathcal{F}_{-} \tau_y f(x)= K_{-}(y,x) \mathcal{F}_{-} f(x), \qquad x \in \mathbb{R}^d.$$
 \end{Def}
By the inversion formula of $ \mathcal{F}_{-}$ the translation can be expressed as an integral operator
$$
\tau_y f(x)= (2 \pi)^{- \frac{d}{2}} \int_{\mathbb{R}^d} (K_{-}( \xi,x))^c K_{-}(y, \xi) \mathcal{F}_{-}f(\xi) \, d \xi.
$$
Since we are working with even dimensions $K_{-}(\xi, x)$ is real valued so we can omit the complex conjugate
\begin{equation}
\label{tra3}
\tau_y f(x)= (2 \pi)^{- \frac{d}{2}} \int_{\mathbb{R}^d} K_{-}( \xi,x) K_{-}(y, \xi) \mathcal{F}_{-}f(\xi) \, d \xi.
\end{equation}
We note that the generalized translation has the following properties, see \cite[Prop. 7.2]{DX}.
\begin{prop}
\label{tra2}
If $d=2$ for all functions $f \in \mathcal{S}(\mathbb{R}^d) \otimes \mathbb{R}_d$ one has
$$\tau_y f(x)=f(x-y).$$
If the dimension $d$ is even and $d>2$ then in general $$ \tau_y f(x) \neq f(x-y).$$
\end{prop}

However, $ \tau_y$ coincides with the classical translation operator if $f$ is a radial function \cite[Thm. 7.3]{DX}.
\begin{prop}
\label{ratra}
Let $f \in \mathcal{S}(\mathbb{R}^d)$ be a radial function on $ \mathbb{R}^d$, $f(x)=f_{0}(|\underline{x}|)$, with $f_{0}: \mathbb{R}_{+} \mapsto \mathbb{R}$, then $ \tau_y f(x)= f(x-y)$.
\end{prop}

Using the generalized translation, it is possible to define a convolution for functions with values in Clifford algebra.

\begin{Def}
\label{cnl}
For $f,g \in \mathcal{S}(\mathbb{R}^d) \otimes \mathbb{R}_d$, the generalized convolution is defined by
$$ (f*_{Cl} g)(x):= (2 \pi)^{-\frac{d}{2}} \int_{\mathbb{R}^d} \tau_y f(x) g(y) \, dy, \qquad x \in \mathbb{R}^d.$$
\end{Def}
We remark that if $f$ and $g$ take value in the Clifford algebra, then $f *_{Cl} g$ is not commutative in general. Moreover, in the case when one of the two functions is radial we have the following well-known property of the Fourier transform \cite[Thm. 8.2]{DX}.
\begin{theorem}
If $g \in \mathcal{S}(\mathbb{R}^d) \otimes \mathbb{R}_d$,  and $f \in \mathcal{S}(\mathbb{R}^d)$ is a radial function, then $f*_{Cl}g$ satisfies
$$ \mathcal{F}_{-}(f*_{Cl}g)(x)=\mathcal{F}_{-}f(x) \mathcal{F}_{-}g(x).$$
In particular, since $f$ is a scalar function we have the commutativity of the convolution, i.e.
$$f*_{Cl}g= g*_{Cl}f.$$
\end{theorem}

Now we have all the tools to build the generalized modulation. Let $f$ be in $\mathcal{S}(\mathbb{R}^d) \otimes \mathbb{R}_d$ and like the generalized translation we fix the following property:
\begin{equation}\label{eq1}
  \mathcal{F}_{-}(M_y f)(\xi)= \tau_y \mathcal{F}_{-}(f)(\xi), \qquad \xi,y \in \mathbb{R}^d,
\end{equation}
where $M_y$ is the generalized modulation operator. By Definition \ref{tra4} we have
$$ \mathcal{F}_{-}(M_y f)(\xi)= \mathcal{F}^{-1}_{-} \bigl(K_{-}(y,x) \mathcal{F}_{-} \bigl( \mathcal{F}_{-}(f))(x)\bigl)(\xi).$$
Now we apply $ \mathcal{F}_{-}$ and use Theorem \ref{mine}, so we are able to give the following definition.
\begin{Def}
\label{mod}
Let $ f \in \mathcal{S}(\mathbb{R}^d) \otimes \mathbb{R}_d$. For $y \in \mathbb{R}^d$ the generalized modulation is defined by
$$ M_{y}f(x)=K_{-}(y,x) f(x), \qquad x \in \mathbb{R}^d.$$
\end{Def}
\begin{nb}
As in the classical Fourier analysis we can relate the Clifford-Fourier transform of the translation with the modulation of the Clifford-Fourier transform:
\begin{equation}\label{prop1}
\mathcal{F}_{-}(\tau_y f(x))=K_{-}(y,x) \mathcal{F}_{-} f(x)=M_{y}(\mathcal{F}_{-}f(x)).
\end{equation}
\end{nb}
\begin{nb}
\label{NC}
In the theory that we are going to develop for even dimensions $d>2$ there is not a commutative relationship between the modulation operator and the translation operator, as it happens in the classical case. Indeed, let $f \in \mathcal{S}(\mathbb{R}^d)$ be a radial function, by Definition \ref{mod} and Proposition \ref{ratra} we get
\begin{equation}
\label{new4}
M_{\omega} \tau_y f(x)=K_{-}(\omega,x)f(x-y), \quad \omega,y \in \mathbb{R}^d.
\end{equation}
Now, if we exchange the roles between the translation and modulation operators by formulas \eqref{tra3} and \eqref{eq1} we have
\begin{eqnarray*}
\tau_{y} M_{\omega} f(x) &=&(2 \pi)^{-\frac{d}{2}} \int_{\mathbb{R}^d} K_{-}(\xi,x) K_{-}(y, \xi) \mathcal{F_{-}}(M_{\omega}f)(\xi) \, d \xi\\
&=& (2 \pi)^{-\frac{d}{2}} \int_{\mathbb{R}^d} K_{-}(\xi,x) K_{-}(y, \xi) \tau_{\omega}\mathcal{F_{-}}(f)(\xi) \, d \xi, \quad \omega,y \in \mathbb{R}^d.
\end{eqnarray*}
Therefore
\begin{equation}
\label{new3}
\tau_{y} M_{\omega} f(x)=(2 \pi)^{-\frac{d}{2}} \int_{\mathbb{R}^d} K_{-}(\xi,x) K_{-}(y, \xi) \tau_{\omega}\mathcal{F_{-}}(f)(\xi) \, d \xi.
\end{equation}
Since \eqref{new4} and \eqref{new3} are very different we do not have any commutative relations between the generalised translation and modulation operators.
\end{nb}

\begin{nb}
One my wonder if for dimension $d=2$ there exists a commutative formula between the generalized translation and modulation operator. One can verify by Proposition \ref{tra2} that for $ f \in \mathcal{S}(\mathbb{R}^d) \otimes \mathbb{R}_d$
$$ M_{\omega} \tau_y f(x)=K_{-}(\omega,x)f(x-y), $$
$$ \tau_{y} M_{\omega} f(x)=K_{-}(\omega,x-y)f(x-y).$$
Now, due to Proposition \ref{linker} and the fact that $K_{-}$ is not symmetric we deduce that it is not possible to make this computation $ K_{-}(\omega,x-y) =K_{-}(\omega,x)K_{-}(\omega,-y)$. Thus, we conclude that never exists a commutative formula between the generalized translation and modulation operator.
\end{nb}

We end this section proving some easy but important formulas of time-frequency analysis.
\begin{lemma}
\label{PF}
If $ f \in \mathcal{S}(\mathbb{R}^d) \otimes \mathbb{R}_d$ and $y, \omega \in \mathbb{R}^d$ then
\begin{equation}
  \label{PF1}
  	\mathcal{F}_{-}(\tau_y M_{\omega} f)(x)= M_{y} \tau_{\omega} \mathcal{F}_{-}f(x), \qquad x \in \mathbb{R}^d,
  \end{equation}
 \begin{equation}
  	\label{PF2}
  	\mathcal{F}_{-}(M_{\omega} \tau_y  f)(x)= \tau_{\omega} M_{y}  \mathcal{F}_{-}f(x) \qquad x \in \mathbb{R}^d.
 \end{equation}
\end{lemma}
\begin{proof}
Formula \eqref{PF1} follows by applying \eqref{prop1} and \eqref{eq1}
 $$\mathcal{F}_{-}(\tau_y M_{\omega} f)(x)= M_{y}\mathcal{F}_{-}( M_{\omega} f)(x)=M_{y} \tau_{\omega} \mathcal{F}_{-}f(x).$$
Formula \eqref{PF2} follows by applying \eqref{eq1} and \eqref{prop1}
$$ \mathcal{F}_{-}(M_{\omega} \tau_y  f)(x)= \tau_{\omega} \mathcal{F}_{-}(\tau_y f)(x)=\tau_{\omega}M_{y} \mathcal{F}_{-}(f)(x).$$

\end{proof}

\section{The Clifford short-time Fourier transform}
The idea of the short-time Fourier transform is to obtain information about local properties of the signal $f$. In order to achieve this aim the signal $f$ is restricted to an interval and after it is evaluated the Fourier transform of the restriction. However, since a sharp cut-off can introduce artificial discontinuities and can create problems, it is usually chosen a smooth cut-off function $g$ called \emph{window function}.

In this section we generalize this concept, using the Clifford-Fourier transform. As signal we consider a function $f$ Clifford-valued and firstly we assume the same hypothesis for the window function $g$. 

\begin{Def}
Let $f,g \in \mathcal{S}(\mathbb{R}^d) \otimes \mathbb{R}_d$. The Clifford short-time Fourier transform of a function $f$ with respect to $g$ is defined as
\begin{equation}\label{CST1}
\mathcal{V}_{g}f(x, \omega)= \mathcal{F}_{-}( \tau_{x} \bar{g} \cdot f)(\omega), \qquad \hbox{for} \quad x, \omega \in \mathbb{R}^d.
\end{equation}
\end{Def}
We want to manipulate formula \eqref{CST1} in order to write it as an integral. From the definition of the Clifford-Fourier transform \eqref{first1} and the formula of the generalized translation operator \eqref{tra3} we get
\begin{eqnarray*}
\mathcal{V}_g f(x, \omega)&=&\mathcal{F}_{-}( \tau_{x} \bar{g} \cdot f)(\omega) =(2 \pi)^{- \frac{d}{2}} \int_{\mathbb{R}^d} K_{-}(t, \omega) \tau_x \bar{g}(t) f(t) \, dt \\
&=& (2 \pi)^{-d} \int_{\mathbb{R}^{2d}}K_{-}(t, \omega) K_{-}(\xi, t) K_{-}(x, \xi) \mathcal{F}_{-}(\bar{g})(\xi) f(t) \, d \xi \, dt   \\
&=& (2 \pi)^{-\frac{3}{2} d} \int_{\mathbb{R}^{3d}}K_{-}(t, \omega) K_{-}(\xi, t) K_{-}(x, \xi)K_{-}(z, \xi) \bar{g}(z) f(t) \, dz \, d \xi \, dt.
\end{eqnarray*}
Since it is difficult to work with this amount of non commuting kernels we choose to work with a radial window function. So if $g \in \mathcal{S}(\mathbb{R}^d)$ is a radial function by definition of the Clifford-Fourier transform and Proposition \ref{ratra} we get
\begin{eqnarray*}
\nonumber
\mathcal{V}_{g}f(x, \omega) &=& (2 \pi)^{- \frac{d}{2}} \int_{\mathbb{R}^{d}} K_{-}(t, \omega) \tau_{x} g(t) f(t) \, dt \\
&=& (2 \pi)^{- \frac{d}{2}} \int_{\mathbb{R}^{d}} K_{-}(t, \omega)g(t-x)f(t) \, dt.
\end{eqnarray*}
Thus we have the following definition.
\begin{Def}
\label{NDE}
Let $f \in \mathcal{S}(\mathbb{R}^d) \otimes \mathbb{R}_d$ and $g \in \mathcal{S}(\mathbb{R}^d)$ be a radial function. The Clifford short-time Fourier transform of a function with respect to $g$ is defined as
\begin{equation}\label{CST2}
\mathcal{V}_{g}f(x, \omega)= \mathcal{F}_{-}( \tau_{x} \bar{g} \cdot f)(\omega)=(2 \pi)^{- \frac{d}{2}} \int_{\mathbb{R}^{d}} K_{-}(t, \omega)g(t-x)f(t) \, dt, \qquad \hbox{for} \quad x, \omega \in \mathbb{R}^d.
\end{equation}
\end{Def}

In the sequel we will use this integral formula for proving all the properties of the Clifford short-time Fourier transform.

Now we are going to show that the Clifford short-time Fourier transform as the Clifford-Fourier transform is well-defined on $B(\mathbb{R}^d) \otimes \mathbb{R}_d$. First of all we recall the following notion
$$ B(\mathbb{R}^d):= \biggl \{ f \in L^1(\mathbb{R}^d): \| f \|_{B}:= \int_{\mathbb{R}^d}(1+|\underline{y}|)^ \lambda |f(y)| \, dy < \infty \biggl \}$$
where $ \lambda= \frac{d-2}{2}$ and $d$ is even more that two. Now, we introduce the following spaces of real valued functions

$$
B^{p}(\mathbb{R}^d):= \biggl \{ f \in L^p(\mathbb{R}^d): \| f \|_{B^p}:= \biggl( \int_{\mathbb{R}^d}(1+|\underline{y}|)^ \lambda |f(y)|^p \, dy \biggl)^{\frac{1}{p}} < \infty \biggl \}\qquad \hbox{for}  \quad 1 \leq p < \infty.
$$

$$
W_{ p \lambda}(\mathbb{R}^d)= \biggl \{ f \in L^{p}(\mathbb{R}^d): \| f\|_{W_{ p\lambda }};= \biggl( \int_{\mathbb{R}^d} (1+|\underline{y}|)^{\lambda p} |f(y)|^{p} \, dy \biggl)^{\frac{1}{p}} < \infty \biggl \} \quad \hbox{for}  \quad 1 \leq p < \infty.
$$
\begin{nb}
If $p=1$ all the spaces introduced coincide.
\end{nb}
For the spaces $B^{p}(\mathbb{R}^d)$ and $W_{ p \lambda}(\mathbb{R}^d)$ we have the following inclusion.
\begin{lemma}
Let $d$ be even more than two. For $p \geq 1$ we have
$$ W_{ p \lambda} (\mathbb{R}^d) \subseteq B^{p}(\mathbb{R}^d).$$
\end{lemma}
\begin{proof}
It is enough to prove that $ \| . \|_{B^p} \leq \| . \|_{W_{ p \lambda}}$. Firstly we observe that since $ \lambda= \frac{d-2}{2} \geq 1$ and $p \geq 1$ we have
$$ (1+|\underline{y}|)^{\lambda} \leq (1+ | \underline{y}| )^{\lambda p}.$$
Then
$$ (1+|\underline{y}|)^{\lambda} |f(y)|^{p} \leq (1+ | \underline{y}| )^{\lambda p} |f(y)|^{p}.$$
Therefore, by the monotonicity of the integral we have
$$ \int_{\mathbb{R}^d}(1+|\underline{y}|)^{\lambda} |f(y)|^{p} \, dy \leq \int_{\mathbb{R}^d} (1+ | \underline{y}| )^{\lambda p} |f(y)|^{p} \, dy.$$
Thus
$$ \| f \|_{B^p} \leq \| f \|_{W_{ p \lambda}}.$$
\end{proof}

\begin{nb}
From the properties of the $L^p$-spaces we do not have any relations of inclusion between the spaces $ B(\mathbb{R}^d)$ and $B^p(\mathbb{R}^d)$. For the same reason there is not any  inclusion between $ B(\mathbb{R}^d)$ and $W_{p \lambda}(\mathbb{R}^d)$.
\end{nb}
Now we prove two inequalities which will be fundamental for defining the domain of the Clifford short-time Fourier transform.
\begin{lemma}
Let $d$ be even more than two. If $f \in L^{2}(\mathbb{R}^d) \otimes \mathbb{R}_d$ and $ g \in W_{2 \lambda }(\mathbb{R}^d)$ is a radial function, then $f \cdot g \in B(\mathbb{R}^d) \otimes \mathbb{R}_d$:
\begin{equation}\label{H1}
\| f \cdot g \|_{B} \leq  c \| g \|_{W_{2 \lambda }} \| f \|_{2},
\end{equation}
where $c$ is a positive constant.
\end{lemma}
\begin{proof}
From the H\"{o}lder inequality we get
\begin{eqnarray*}
\| f \cdot g \|_{B}& \leq & c \int_{\mathbb{R}^d}(1+|\underline{y}|)^{\lambda}|f(y)| |g(y)| \, dy = c \int_{\mathbb{R}^d}(1+|\underline{y}|)^{\lambda}|g(y)| |f(y)| \, dy \\
& \leq & c \biggl( \int_{\mathbb{R}^d} (1+|\underline{y}|)^{ 2 \lambda}|g(y)|^2 \, dy \biggl)^{\frac{1}{2}} \biggl( \int_{\mathbb{R}^d} |f(y)|^2 \, dy \biggl)^{\frac{1}{2}} \\
&=& c \| g \|_{W_{2 \lambda }} \| f \|_{2}.
\end{eqnarray*}
\end{proof}

\begin{prop}
Let $d$ be even more than two. If $ f \in B^{2}(\mathbb{R}^d) \otimes \mathbb{R}_d$ and $ g \in B^{2}(\mathbb{R}^d)$ is a radial function, then $f \cdot g \in B(\mathbb{R}^d) \otimes \mathbb{R}_d $:
\begin{equation}\label{H2}
\| f \cdot g \|_{B} \leq c \| f \|_{B^{2}} \| g \|_{B^{2}},
\end{equation}
where $c$ is a positive constant.
\end{prop}
\begin{proof}
From the H\"{o}lder inequality we get
\begin{eqnarray*}
\| f \cdot g \|_{B} & \leq &  c \int_{\mathbb{R}^d}(1+|\underline{y}|)^{\lambda}|f(y)| |g(y)| \, dy = c \int_{\mathbb{R}^d}(1+|\underline{y}|)^{\frac{\lambda}{2}}|f(y)| (1+|\underline{y}|)^{\frac{\lambda}{2}} |g(y)|  \, dy \\
&\leq & c \biggl(\int_{\mathbb{R}^d} (1+|\underline{y}|)^{\lambda} |f(y)|^2  \, dy \biggl)^{\frac{1}{2}} \biggl(\int_{\mathbb{R}^d} (1+|\underline{y}|)^{\lambda} |g(y)|^2 dy \biggl)^{\frac{1}{2}} \, \\
&=& c \| f \|_{B^{2}} \| g \|_{B^{2}}.
\end{eqnarray*}
\end{proof}
Now we show that the space $W_{ p \lambda}(\mathbb{R}^d)$ is invariant under translation of radial functions.
\begin{lemma}
\label{Inva}
Let $d$ be even more than two and $p \geq 1$. If $g$ is a radial function in $W_{ p \lambda}(\mathbb{R}^d)$ then $ \tau_x g \in W_{ p \lambda}(\mathbb{R}^d)$, i.e
\begin{equation}\label{Inv2}
 \| \tau_{x} g \|_{W_{ p \lambda}} \leq (1+|\underline{x}|)^{\lambda} \| g \|_{W_{ p \lambda}}, \qquad x \in \mathbb{R}^d.
\end{equation}
\end{lemma}
\begin{proof}
Since $g$ is a radial function by Proposition \ref{ratra} we have
$$ \| \tau_{x} g \|_{W_{ p \lambda}}^p= \int_{\mathbb{R}^d} (1+|\underline{t}|)^{p \lambda} | \tau_x g(t)|^p \, dt=\int_{\mathbb{R}^d} (1+|\underline{t}|)^{p \lambda} |g(t-x)|^p \, dt.$$
Now, we put $t-x=z$ and by the triangle inequality we obtain
\begin{eqnarray*}
 \| \tau_{x} g \|_{W_{ p \lambda}}^p &=& \int_{\mathbb{R}^d} (1+|\underline{z}+ \underline{x}|)^{p \lambda} |g(z)|^p \, dt  \leq \int_{\mathbb{R}^d} (1+|\underline{z}|+ |\underline{x}|)^{p \lambda} |g(z)|^p \, dz \\
 &\leq&   \int_{\mathbb{R}^d} (1+|\underline{z}|)^{p \lambda}(1+|\underline{x}|)^{p \lambda} |g(z)|^p \, dz = (1+|\underline{x}|)^{p \lambda} \int_{\mathbb{R}^d} (1+|\underline{z}|)^{p \lambda} |g(z)|^p \, dz   \\
 &=& (1+|\underline{x}|)^{p \lambda} \| g \|_{W_{ p \lambda}}^p.
\end{eqnarray*}
So we gain the thesis.
\end{proof}
\begin{nb}
Using the same techniques and the hypothesis of Lemma \ref{Inva} it is possible to prove that the space $B^p(\mathbb{R}^d)$ is invariant under translation, too. Indeed 
\begin{equation} \label{Inv1}
\| \tau_{x} g \|_{B^{p}} \leq (1+|\underline{x}|)^{\frac{\lambda}{p}} \| g \|_{B^{p}}, \qquad x \in \mathbb{R}^d.
\end{equation}
\end{nb}
Now, we have all the tools for proving that $ \mathcal{V}_{g}f \in B(\mathbb{R}^d) \otimes \mathbb{R}_d$.
\begin{theorem}
Let $d>2$ and even.
If $f \in B^2(\mathbb{R}^d) \otimes \mathbb{R}_d$ and $g \in B^2(\mathbb{R}^d)$ is a radial function, then the Clifford short-time Fourier transform is well defined.
\end{theorem}
\begin{proof}
If we prove that $ \tau_x g \cdot f \in B(\mathbb{R}^d) \otimes \mathbb{R}_d$, then by Theorem \ref{WD} and Definition \ref{NDE} we have the thesis.
\newline
So, we focus on proving that $ \tau_x g \cdot f \in B(\mathbb{R}^d) \otimes \mathbb{R}_d$. By inequalities \eqref{H2} and \eqref{Inv1} we have
$$ \| \tau_x g \cdot f \|_{B} \leq c \| \tau_x g \|_{B^2} \| f \|_{B^ 2} \leq c(1+| \underline{x}|)^{\frac{\lambda}{2}} \| g\|_{B^2} \| f \|_{B^ 2} < \infty,$$
where $c$ is a positive constant.
\end{proof}
It is possible to prove the well-posedness of the Clfford short-time Fourier transform choosing different spaces for the signal $f$ and the window function $g$.
\begin{theorem}
Let $d>2$ and even. If $f \in L^2(\mathbb{R}^d) \otimes \mathbb{R}_d$ and $g \in W_{2 \lambda }(\mathbb{R}^d)$ is a radial function, then the Clifford short-time Fourier transform is well defined.
\end{theorem}
\begin{proof}
As before it is an application of Theorem \ref{WD}. So we focus on proving that $ \tau_x g \cdot f \in B(\mathbb{R}^d) \otimes \mathbb{R}_d$. By inequalities \eqref{H1} and \eqref{Inv2} we have
$$ \| \tau_x g \cdot f \|_{B} \leq c \| \tau_x g \|_{W_{2 \lambda }} \| f \|_{2} \leq c (1+| \underline{x}|)^{\lambda} \| g\|_{W_{2 \lambda }} \| f \|_{2} < \infty,$$
where $c$ is positive constant.
\end{proof}
In the rest of the paper we consider (except some cases) the signal $f \in L^{2}(\mathbb{R}^d) \otimes \mathbb{R}
_d$ and the window function $g \in W_{2 \lambda }(\mathbb{R}^d)$ a radial function.

\section{Elementary properties of Clifford short-time Fourier transform}
In this section we prove some basic properties of the Clifford short-time Fourier transform.
\begin{prop}
Let $d>2$ and even. If $f \in L^2(\mathbb{R}^d) \otimes \mathbb{R}_d$ and $g \in W_{2 \lambda }(\mathbb{R}^d)$ is a radial function then
\begin{itemize}
  \item[1](Right linearity) If $\lambda, \mu \in \mathbb{R}_d$ and $h \in L^{2}(\mathbb{R}^d) \otimes \mathbb{R}_d $ then
  $$ [\mathcal{V}_{g}(f \lambda + h \mu)](\omega,x)=\mathcal{V}_{g}(f)(x, \omega) \lambda + \mathcal{V}_{g}(h)(x, \omega) \mu.$$
  \item[2](Parity)
$$ \mathcal{V}_{g}f(x, \omega)=\mathcal{V}_{g}f(-x, \omega).$$
\end{itemize}
\end{prop}
\begin{proof}
The first one follows from the Definition \ref{NDE}. The second one follows from the hypothesis of radiality of $g$.
\end{proof}

In the next Proposition we lists some equivalent forms of the Clifford short-time Fourier transform.

\begin{prop}
\label{list}
Let $d>2$ and even. We suppose that $g \in W_{2 \lambda }(\mathbb{R}^d)$ is a radial function and $f \in L^2(\mathbb{R}^d) \otimes \mathbb{R}_d$. Then
\begin{enumerate}
\item
\begin{equation}
\label{imp}
\mathcal{V}_{g}f(x,\omega)= (2 \pi)^{- \frac{d}{2}} \int_{\mathbb{R}^d} \overline{M_{\omega} \tau_x g(t)} f(t) \, dt,
\end{equation}
\item  
\begin{equation}
\label{F2}
\mathcal{V}_{g}f(x,\omega)=(2 \pi)^{- \frac{d}{2}} \int_{\mathbb{R}^d} \overline{ \tau_{ \omega}  M_{x} (\mathcal{F}_{-}g)(t)} (\mathcal{F}_{-}f)(t) \, dt, 
\end{equation}
\item
\begin{equation}
\label{FI}
\mathcal{V}_{g}f(x,\omega)= \mathcal{V}_{\mathcal{F}_{-}(g)} \mathcal{F}_{-}(f)(\omega,x)- (2 \pi)^{- \frac{d}{2}}\int_{\mathbb{R}^d}\overline{[\tau_x, M_{\omega}] g(t)} f(t) \, dt,
\end{equation}
\item
\begin{equation}
\label{F5}
\mathcal{V}_{g}f(x,\omega)= \mathcal{F}_{-} \bigl( \mathcal{F}_{-}(f) \cdot \tau_{\omega} \mathcal{F}_{-}(g) \bigl)(x) - (2 \pi)^{- \frac{d}{2}}\int_{\mathbb{R}^d}\overline{[\tau_x, M_{\omega}] g(t)} f(t) \, dt,
\end{equation}
\end{enumerate}
where $[. \, ,.]$ is the commutator defined in \eqref{comm1}.
\end{prop}
\begin{proof}
\begin{enumerate}
\item
To prove the equality \eqref{imp} we use Definition \ref{NDE}, the relation \eqref{inv} and the radiality of the function $g$ 
\begin{eqnarray*}
\mathcal{V}_{g} f(x, \omega) &=& (2 \pi)^{- \frac{d}{2}}\int_{\mathbb{R}^d} K_{-}(t, \omega) \tau_x g(t) f(t) \, dt \\
&=& (2 \pi)^{- \frac{d}{2}} \int_{\mathbb{R}^d} \overline{K_{-}(\omega, t) \tau_{x} g(t)} f(t) \, dt \\
&=& (2 \pi)^{- \frac{d}{2}} \int_{\mathbb{R}^d} \overline{ M_{\omega} \tau_{x}g(t)} f(t) \, dt.
\end{eqnarray*}
\item We prove \eqref{F2} using \eqref{imp}, Plancherel's theorem (see Proposition \ref{Plan}) and the formula \eqref{PF2} 
\begin{eqnarray*}
 \mathcal{V}_{g} f(x, \omega) &=& (2 \pi)^{- \frac{d}{2}} \int_{\mathbb{R}^d} \overline{(\mathcal{F}_{-} \bigl( M_{\omega} \tau_{x} g \bigl)(y)}  \mathcal{F}_{-}(f)(y) \, dy\\
 &=& (2 \pi)^{- \frac{d}{2}} \int_{\mathbb{R}^d} \overline{\tau_{\omega} M_{x} \mathcal{F}_{-} (g)(y)} \mathcal{F}_{-}(f)(y) \, dy.
\end{eqnarray*}

\item
To show \eqref{FI} we observe that since $g$ is a radial function also its Clifford-Fourier transform is radial. Thus by relation \eqref{inv} we have
\begin{eqnarray*}
\mathcal{V}_{\mathcal{F}_{-}(g) }\mathcal{F}_{-}(f)(\omega, x)  &=& (2 \pi)^{- \frac{d}{2}} \int_{\mathbb{R}^d}K_{-}(t,x) \tau_{\omega} \mathcal{F}_{-}g(t) \mathcal{F}_{-}f(t) \, dt\\
&=& (2 \pi)^{- \frac{d}{2}} \int_{\mathbb{R}^d} \overline{K_{-}(x,t) \tau_{\omega} \mathcal{F}_{-}g(t)} \mathcal{F}_{-}f(t) \, dt\\
&=&  (2 \pi)^{- \frac{d}{2}} \int_{\mathbb{R}^d} \overline{M_x \tau_{\omega} \mathcal{F}_{-}g(t)} \mathcal{F}_{-}f(t) \, dt.
\end{eqnarray*}
Finally, using Plancherel's theorem, Theorem \ref{mine} and the relation \eqref{PF2} we obtain
\begin{eqnarray}
\nonumber
\mathcal{V}_{\mathcal{F}_{-}(g) }\mathcal{F}_{-}(f)(\omega,x)  &=& (2 \pi)^{- \frac{d}{2}}  \int_{\mathbb{R}^d} \overline{\mathcal{F}_{-} \bigl(M_x \tau_{\omega} \mathcal{F}_{-}g(t)\bigl)} f(t) \, dt \\ \label{St1}
&=& (2 \pi)^{- \frac{d}{2}} \int_{\mathbb{R}^d} \overline{\tau_{x} M_{ \omega} g(t)} f(t) \, dt. 
\end{eqnarray}

Since the generalized translation and generalized modulation does not commute (see Remark \ref{NC}) we cannot exchange the rules of $\tau_{x}$ and $M_{\omega}$. To change the order we use the commutator defined in \eqref{comm1}, so we can relate this formula with the Clifford short-time Fourier transform of $f$ with respect to $g$ by using \eqref{imp}:
\begin{eqnarray*}
\mathcal{V}_{\mathcal{F}_{-}(g) }\mathcal{F}_{-}(f)(\omega, x) &=&  (2 \pi)^{- \frac{d}{2}}  \int_{\mathbb{R}^d} \overline{[\tau_{x}, M_{\omega}] g(t)} f(t) \, dt+ (2 \pi)^{- \frac{d}{2}}   \int_{\mathbb{R}^d} \overline{M_{\omega} \tau_{x} g(t)} f(t) \, dt \\
&=&  (2 \pi)^{- \frac{d}{2}}  \int_{\mathbb{R}^d} \overline{[\tau_{x}, M_{\omega}] g(t)} f(t) \, dt+ \mathcal{V}_{g}f(x, \omega).
\end{eqnarray*}
\item 
Finally formula \eqref{F5} follows from the relations \eqref{inv} and the equality  \eqref{imp}
\begin{eqnarray*}
\mathcal{F}_{-}\bigl((\mathcal{F}_{-}(f) \cdot \tau_{\omega} \mathcal{F}_{-}(g))(x)&=& (2 \pi)^{- \frac{d}{2}} \int_{\mathbb{R}^d}K_{-}(y,x) \mathcal{F}_{-}(f)(y) \tau_{\omega}\mathcal{F}_{-}(g)(y) \, dy \\
&=& (2 \pi)^{- \frac{d}{2}} \int_{\mathbb{R}^d}\overline{K_{-}(x,y) \tau_{\omega}\mathcal{F}_{-}(g)(y) }\mathcal{F}_{-}(f)(y) \, dy\\
&=&  (2 \pi)^{- \frac{d}{2}}\int_{\mathbb{R}^d}\overline{M_{x} \tau_{\omega}\mathcal{F}_{-}(g)(y) }\mathcal{F}_{-}(f)(y) \, dy\\
&=& \mathcal{V}_{\mathcal{F}_{-}(g)} \mathcal{F}_{-}(\omega, x).\\
\end{eqnarray*}
Using formula \eqref{FI} we obtain the thesis.
\end{enumerate}
\end{proof}

\begin{nb}
The formulas proved in Proposition \ref{list} are similar to the classical case ( see \cite[Lemma 3.1.1]{G}). The main difference is the presence of the following integral
$$\int_{\mathbb{R}^d}\overline{[\tau_x, M_{\omega}] g(t)} f(t) \, dt.$$
This is due to the lack of commutativity.
\end{nb}

\begin{nb}
Another difference with respect to the classical case is that it is not possible to write the Clifford short-Fourier transform as a convolution of the Clifford-Fourier transfrom of the signal and the Clifford-Fourier transform of the window function. For example it is not possible to prove a formula like this
$$ \mathcal{V}_g f(x, \omega)=(\overline{M_x \mathcal{F_{-}}(g)} *_{Cl} \mathcal{F_{-}}(f))(\omega).$$ 
Indeed, by the definition of convolution (see Definition \ref{cnl}) we have
$$ (\overline{M_x \mathcal{F_{-}}(g)}*_{Cl} \mathcal{F_{-}}(f))(\omega)=(2 \pi)^{- \frac{d}{2}} \int_{\mathbb{R}^d} \overline{\tau_y M_x \mathcal{F_{-}}(g)(\omega) \mathcal{F_{-}}(f)(y)} \, dy.$$
Now, it is not possible to compute $\tau_y M_x \mathcal{F_{-}}(g)$ using the ordinary formula of the translation because $M_x \mathcal{F_{-}}(g)$ is no longer radial and so it is not possible to derive a relation with the Clifford short-time Fourier transform.
\end{nb}
Now we study the continuity of the Clifford short-time Fourier transform. 
\begin{theorem}
Let $d>2$ and even. If $g \in W_{2 \lambda }(\mathbb{R}^d)$ is a radial function and $f \in L^2(\mathbb{R}^d) \otimes \mathbb{R}_d$ then the Clifford short-time Fourier transform is a continuous operator on $L^{2}(\mathbb{R}^d) \otimes \mathbb{R}_d$.
\end{theorem}
\begin{proof}
We remark that $g \in W_{2 \lambda }(\mathbb{R}^d) \subseteq L^{2}(\mathbb{R}^d)$ hence by the formula \eqref{imp} it is enough to prove the continuity of the operators $\tau_x$ and $M_{\omega} $. The continuity of the translation operator follows from the following fact
$$ \lim_{x \to 0} \| \tau_x g - g \|_2=0.$$
On the other side the continuity of the modulation operator follows from the formula \eqref{eq1} and the Parseval's identity (see Proposition \ref{Par})
$$ \lim_{\omega  \to 0} \| M_{\omega} g-g \|_2= \lim_{\omega \to 0} \| \tau_{\omega} \mathcal{F_{-}}(g)-\mathcal{F_{-}}(g) \|_2=0.$$
\end{proof}

Before to state the next result, we have to introduce some notations. We call $ f \otimes g$ the tensor product between $f$ and $g$ and  it acts in the following way
$$ (f \otimes g)(x,t)=f(x) g(t).$$
Let $ \mathcal{T}$ be the asymmetric coordinate transform of a function $f$ on $ \mathbb{R}^{2d}$
\begin{equation}
\label{anti}
\mathcal{T} f(x,t)=f(t,t-x).
\end{equation}
\begin{Def}[Partial Clifford-Fourier transform]
Let $f \in  \mathcal{S}(\mathbb{R}^{2d}) \otimes \mathbb{R}_d$. We define the partial Clifford-Fourier transform in the following way
\begin{equation}
\mathcal{F}_{2^{-}} f(x, \omega)=(2 \pi)^{- \frac{d}{2}} \int_{\mathbb{R}^d} K_{-}(t, \omega) f(x,t) \, dt.
\end{equation}
\end{Def}
\begin{nb}
The partial Clifford-Fourier transform is the Cliffrod Fourier transform defined in \eqref{first1} with respect to the second variable, hence the variable $x$ is considered as a parameter.
\end{nb}

\begin{nb}
All the properties which hold for the Clifford-Fourier transform as the Plancherel's theorem and the Parseval's identity stay true also for the partial Clifford-Fourier transform.
\end{nb}

Now, we show another way to write the Clifford short-time Fourier transform using the tensor product and the partial Clifford-Fourier transform, as in the classical case \cite{G}.
\begin{lemma}
\label{tensor}
If $g \in W_{2 \lambda }(\mathbb{R}^d)$ is a radial function and $f \in L^{2}(\mathbb{R}^d) \otimes  \mathbb{R}_d$ then
$$ \mathcal{V}_{g}f(x, \omega)= \mathcal{F}_{2^{-}} \bigl( \mathcal{T}(f \otimes g)(x, \omega) \bigl).$$
\end{lemma}
\begin{proof}
From the definition of the partial Clifford-Fourier transform and the definition of $\mathcal{T}$ (see formula \eqref{anti}) we obtain
\begin{eqnarray*}
\mathcal{F}_{2^{-}} \bigl( \mathcal{T}(f \otimes g)(x, \omega) \bigl) &= & \int_{\mathbb{R}^d} K_{-}(t, \omega)  \mathcal{T}\bigl (f \otimes g \bigl )(x, t) \, dt=\int_{\mathbb{R}^d} K_{-}(t, \omega)  \mathcal{T}\bigl (f(x) g(t) \bigl) \, dt\\
&=&\int_{\mathbb{R}^d} K_{-}(t, \omega) f(t) g(t-x) \, dt.
\end{eqnarray*}
Now, using the fact that the window function $g$ is real, thus can commute, and Definition \ref{NDE} we get
$$ \mathcal{F}_{2^{-}} \bigl( \mathcal{T}(f \otimes g)(x, \omega) \bigl)= \int_{\mathbb{R}^d} K_{-}(t, \omega) f(t) g(t-x)= \int_{\mathbb{R}^d} K_{-}(t, \omega) g(t-x) f(t) \, dt\\
=\mathcal{V}_{g}f(x, \omega).$$
\end{proof}
Now, we prove a sort of "covariance" property. The main difference from the classical case (see \cite[Lemma 3.1.3]{G}) is that we do not have an equality, this is due to the lack of commutativity.
\begin{lemma}
Let $d>2$ and even. We suppose that $g \in W_{2 \lambda }(\mathbb{R}^d)$ is a radial function and $f \in L^2(\mathbb{R}^d) \otimes \mathbb{R}_d$, thus we have
$$ | \mathcal{V}_{g}f(x- \mu, \omega- \eta)| \leq c(1+|\underline{\omega}|)^{\lambda} (1+|\underline{\eta}|)^{\lambda} (1+|\underline{x}|)^{\lambda} (1+|\underline{\mu}|)^{ \lambda} \| f \|_{2} \| g \|_{W_{2 \lambda }}, \qquad \hbox{for} \quad x, \mu, \omega, \eta \in \mathbb{R}^d,$$
where $c$ is a positive constant.
\end{lemma}

\begin{proof}
Firstly we use Lemma \ref{Ker1}
\begin{eqnarray*}
| \mathcal{V}_{g}f(x- \mu, \omega- \eta)| & \leq & c \int_{\mathbb{R}^d} |K_{-}(t, \omega- \eta)|| f(t) g(t-x + \mu)| \, dt \\
& \leq & c \int_{\mathbb{R}^d} (1+| \underline{t}|)^{\lambda}(1+ |\underline{\omega}- \underline{\eta}|)^{\lambda} | f(t) g(t-x + \mu)| \, dt \\
& \leq& c (1+ |\underline{\omega}|)^{\lambda} (1+ |\underline{\eta}|)^{\lambda}  \int_{\mathbb{R}^d} (1+| \underline{t}|)^{\lambda} | f(t)| |g(t-x + \mu)| \, dt,
\end{eqnarray*}
where $c$ is a positive constant. Now, by H\"{o}lder inequality we get
$$ | \mathcal{V}_{g}f(x- \mu, \omega- \eta)|  \leq c  (1+ |\underline{\omega}|)^{\lambda} (1+ |\underline{\eta}|)^{\lambda}  \| f \|_{2} \biggl( \int_{\mathbb{R}^d} (1+|\underline{t}|)^{2 \lambda} |g(t-x + \mu)|^2 \, dt \biggl)^{\frac{1}{2}}.$$
Putting $t-x+ \mu=z$ in the integral we obtain
\begin{eqnarray*}
| \mathcal{V}_{g}f(x- \mu, \omega- \eta)|  &\leq&   \!  \! \! \! c (1+ |\underline{\omega}|)^{\lambda} (1+| \underline{\eta}|)^{\lambda} \| f \|_{2} \biggl( \int_{\mathbb{R}^d} (1+|\underline{z}+\underline{x}- \underline{\mu}|)^{2 \lambda} |g(z)|^2 \, dz
\biggl)^{\frac{1}{2}}\\
 &\leq& \!  \! \! \!  c (1+ |\underline{\omega}|)^{\lambda} (1+| \underline{\eta}|)^{\lambda} \| f \|_{2} \biggl( \int_{\mathbb{R}^d} (1+|\underline{z}|+|\underline{x}|+ |\underline{\mu}|)^{2 \lambda} |g(z)|^2 \, dz
\biggl)^{\frac{1}{2}}\\
& \leq & \!  \! \! \!  c (1+ |\underline{\omega}|)^{\lambda} (1+| \underline{\eta}|)^{\lambda} (1+|\underline{x}|)^{\lambda} (1+|\underline{\mu}|)^{\lambda} \| f \|_{2} \biggl( \int_{\mathbb{R}^d} (1+|\underline{z}|)^{2 \lambda} |g(z)|^2 \, dz \biggl)^{\frac{1}{2}}\\
&=&  \!  \! \! \! c(1+|\underline{\omega}|)^{\lambda} (1+|\underline{\eta}|)^{\lambda} (1+|\underline{x}|)^{\lambda} (1+|\underline{\mu}|)^{ \lambda} \| f \|_{2} \| g \|_{W_{2 \lambda }}.
\end{eqnarray*}
\end{proof}

\section{Modulation and translation of the signal and of the window}
In this section we discuss what happens if we modulate and translate the signal and the window function, respectively. Surprisingly, there are some differences with respect to the classical Fourier analysis: the estimates depends on the Clifford-Fourier transform of $f$ and the convolution between a function and $g$.

\begin{prop}
\label{A0}
Let $d>2$ and even. If $ f \in \mathcal{S}(\mathbb{R}^d) \otimes \mathbb{R}_d$ and $g$ is a radial function in $\mathcal{S}(\mathbb{R}^d)$, then for $ \theta, \eta, \mu, x, \omega \in \mathbb{R}^d$ we have
$$
|\mathcal{V}_{\tau_{\theta} g} \tau_{\mu} M_{\eta} f (x, \omega)| \leq c (1+|\underline{\omega}|)^{\lambda} (1+|\underline{\mu}|)^{\lambda} (1+|\underline{\theta}|)^{2 \lambda} \left( (1+|.|)^{2 \lambda} * |g| \right)(x)\int_{\mathbb{R}^d} (1+ |.|)^{2 \lambda} |\tau_{\eta} \mathcal{F_{-}}f(.)| d \cdot 
$$
where $.$ is a fixed variable and $c$ is a positive constant.
\end{prop}
\begin{proof}
By Definition \ref{NDE} we get
\begin{equation}
\label{ref0}
 |\mathcal{V}_{\tau_{\theta} g} \tau_{\mu} M_{\eta} f (x, \omega)| \leq c	\int_{\mathbb{R}^d}|K_{-}(t, \omega)| |\tau_{\theta} g(t-x)| | \tau_{\mu}M_{\eta} f(t) | \, dt.
\end{equation}
Since by hypothesis the function $g$ is radial to compute the translation of the function $ g(t-x)$ we can use the ordinary formula of the translation (see Proposition \ref{ratra}), thus
\begin{equation}
\label{ref1}
\tau_{\theta} g(t-x)=g(t-x- \theta).
\end{equation}
On the other hand, in order to compute  $\tau_{\mu} M_{\eta} f(t)$, since we are not translating a radial function, we have to use the formula \eqref{tra3} and the relation \eqref{eq1}
\begin{eqnarray*}
\tau_{\mu} M_{\eta} f(t) &=& (2 \pi)^{-\frac{d}{2}} \int_{\mathbb{R}^d} K_{-}(\xi,t) K_{-}(\mu, \xi) \mathcal{F_{-}}(M_{\eta}f)(\xi) \, d\xi\\
&=& (2 \pi)^{-\frac{d}{2}} \int_{\mathbb{R}^d} K_{-}(\xi,t) K_{-}(\mu, \xi)  \tau_{\eta}(\mathcal{F_{-}}f)(\xi) \, d\xi.
\end{eqnarray*}
Therefore
\begin{equation}
\label{ref2}
\tau_{\mu} M_{\eta} f(t)=(2 \pi)^{-\frac{d}{2}} \int_{\mathbb{R}^d} K_{-}(\xi,t) K_{-}(\mu, \xi)  \tau_{\eta}(\mathcal{F_{-}}f)(\xi) \, d\xi.
\end{equation}	
Putting \eqref{ref1} and \eqref{ref2} in \eqref{ref0} we obtain
$$ |\mathcal{V}_{\tau_{\theta}  g} \tau_{\mu} M_{\eta} f (x, \omega)| \leq  c \int_{\mathbb{R}^{2d}}|K_{-}(t, \omega)|  |g(t-x- \theta)| |K_{-}(\xi, t)| |K_{-}(\mu, \xi)| |\tau_{\eta} \mathcal{F_{-}}f(\xi)| \, d \xi \, dt.$$
Using the upper bound of the kernel $K_{-}$ (see Lemma \ref {Ker1}) we get
$$ |\mathcal{V}_{\tau_{\theta}  g} \tau_{\mu} M_{\eta} f (x, \omega)| \leq  c (1+ |\underline{\omega}|)^{\lambda} (1+ |\underline{\mu}|)^{\lambda} \int_{\mathbb{R}^{2d}} (1+ |\underline{t}|)^{2 \lambda} (1+ |\underline{\xi}|)^{2 \lambda}|g(t-x- \theta)| |\tau_{\eta}\mathcal{F_{-}}f(\xi)| \, d \xi \, dt.$$
Now, we put $z=t- \theta$
\begin{eqnarray*}
|\mathcal{V}_{\tau_{\theta}  g} \tau_{\mu} M_{\eta} f (x, \omega)| &\leq&  \! \! \! c (1+ |\underline{\omega}|)^{\lambda} (1+ |\underline{\mu}|)^{\lambda} \! \int_{\mathbb{R}^{2d}} (1+ |\underline{z}+\underline{\theta}|)^{2 \lambda} (1+ |\underline{\xi}|)^{2 \lambda}|g(z-x)| |\tau_{\eta}\mathcal{F_{-}}f(\xi)| d \xi dz\\
& \leq& \! \! \! c (1+ |\underline{\omega}|)^{\lambda} (1+ |\underline{\mu}|)^{\lambda}(1+ |\underline{\theta}|)^{2 \lambda} \int_{\mathbb{R}^{2d}} (1+ |\underline{z}|)^{2 \lambda}|g(z-x)| \cdot \\
&& \cdot (1+ |\underline{\xi}|)^{2 \lambda} |\tau_{\eta}\mathcal{F_{-}}f(\xi)| \, d \xi \, dz.
\end{eqnarray*}
Now, since $g$ is radial we have the following equality
\begin{equation}
\label{ref6}
g(z-x)=g(x-z)=\tau_z g(x).
\end{equation}
By Fubini' theorem and the definition of convolution we get
\begin{eqnarray*}
|\mathcal{V}_{\tau_{\theta}  g} \tau_{\mu} M_{\eta} f (x, \omega)| &\leq&  c (1+ |\underline{\omega}|)^{\lambda} (1+ |\underline{\mu}|)^{\lambda}(1+ |\underline{\theta}|)^{2 \lambda} \left( (1+|z|)^{2 \lambda} * |g| \right)(x)\\
&& \cdot \int_{\mathbb{R}^d} (1+|\underline{\xi}|)^{2 \lambda}|\tau_{\eta}\mathcal{F_{-}}f(\xi)| \, d \xi.
\end{eqnarray*}
\end{proof}

We observed that in Clifford-Fourier analysis there is not a commutative relations between the modulation and the translation operator (see Remark \ref{NC}). This is confirmed by the following estimate where we exchange the order of the modulation and translation in the signal.
\begin{prop}
\label{A2}
Let $d>2$ and even. If $ f \in \mathcal{S}(\mathbb{R}^d) \otimes \mathbb{R}_d$ and $g$ is a radial function in $\mathcal{S}(\mathbb{R}^d)$, then for $ \theta, \eta, \mu, x, \omega \in \mathbb{R}^d$ we have
$$
|\mathcal{V}_{\tau_{\theta}  g} M_{\eta} \tau_{\mu}  f (x, \omega)| \leq \!  c (1+|\underline{\omega}|)^{\lambda} (1+|\underline{\eta}|)^{\lambda} (1+|\underline{\mu}|)^{\lambda}  (1+|\underline{\theta}|)^{3 \lambda} \left( \!  (1+|.|)^{3 \lambda}\!  * \!|g| \right)\!(x) \! \! \int_{\mathbb{R}^d} \! \! (1+ |.|)^{2 \lambda} |\mathcal{F_{-}}f(.)| d \cdot 
$$
where $.$ is a fixed variable and $c$ is a positive constant.
\end{prop}
\begin{proof}
By Definition \ref{NDE} and the radiality of $g$ we get
\begin{equation}
\label{ref4}
|\mathcal{V}_{\tau_{\theta}  g} M_{\eta} \tau_{\mu}  f (x, \omega)| \leq c \int_{\mathbb{R}^d} |K_{-}(t, \omega)| | g(t-x-\theta)||M_{\eta} \tau_{\mu} f(t)|\, dt.
\end{equation}
To translate the function $f$ we need the formula \eqref{tra3} (because $f$ is not radial), thus
\begin{equation}
\label{ref5}
M_{\eta} \tau_{\mu} f(t)=K_{-}(\eta,t) \tau_{\mu}f(t)=K_{-}(\eta,t) \int_{\mathbb{R}^d} K_{-}(\xi,t)K_{-}(\mu, \xi) \mathcal{F_{-}} f(\xi) \, d \xi.
\end{equation}
Putting \eqref{ref5} in \eqref{ref4} we obtain
$$ |\mathcal{V}_{\tau_{\theta}  g} M_{\eta} \tau_{\mu}  f (x, \omega)| \leq c \int_{\mathbb{R}^{2d}} |K_{-}(t, \omega)| | g(t-x-\theta)| |K_{-}(\eta,t)| |K_{-}(\xi,t)||K_{-}(\mu, \xi)|| \mathcal{F_{-}} f(\xi)| \, d \xi dt.$$
Now by the upper bound of the kernel (see Lemma \ref{Ker1}) we get

$$ |\mathcal{V}_{\tau_{\theta}  g} M_{\eta} \tau_{\mu}  f (x, \omega)| \leq c (1+ |\underline{\omega}|)^{\lambda}(1+ |\underline{\eta}|)^{\lambda} (1+ |\underline{\mu}|)^{\lambda}
\int_{\mathbb{R}^{2d}} (1+ |\underline{t}|)^{ 3 \lambda} (1+ |\underline{\xi}|)^{2 \lambda} |g(t-x-\theta)| |\mathcal{F_{-}}f(\xi)| \, d \xi dt.$$
Putting $z=t- \theta$, using the equality \eqref{ref6} and the Fubini's theorem we obtain
\begin{eqnarray*}	
|\mathcal{V}_{ \tau_{\theta}  g}  M_{\eta} \tau_{\mu} f (x, \omega)| &\leq& c  (1+ |\underline{\omega}|)^{\lambda}(1+ |\underline{\eta}|)^{\lambda} (1+ |\underline{\mu}|)^{\lambda}
\int_{\mathbb{R}^{2d}} (1+ |\underline{z}+\underline{\theta}|)^{ 3 \lambda} (1+ |\underline{\xi}|)^{2 \lambda}\cdot \\
&& \cdot |g(z-x)| |\mathcal{F_{-}}f(\xi)| \, d \xi dz\\
&=& c (1+|\underline{\omega}|)^{\lambda} (1+|\underline{\eta}|)^{\lambda} (1+|\underline{\mu}|)^{\lambda}  (1+|\underline{\theta}|)^{3 \lambda} \left( (1+|z|)^{3 \lambda} * |g| \right)(x) \cdot \\
&& \cdot \int_{\mathbb{R}^d}(1+ |\xi|)^{2 \lambda} |\mathcal{F_{-}}f(\xi)| \, d \xi. 
\end{eqnarray*}
\end{proof}
\begin{nb}
One can wonder if it is possible to have an estimate for 
\begin{equation}
\label{Rs}
| \mathcal{V}_{M_{q} \tau_{\theta} g} M_{\eta} \tau_{\mu} f(x, \omega)|, \qquad q, \theta, \eta, \mu, x, \omega \in \mathbb{R}^d.
\end{equation}
Basically with respect to Proposition \ref{A2} we make the modulation of the window function. However, it is not possible to have an estimate for \eqref{Rs}. From Definition \ref{NDE} we have
\begin{equation}
\label{R0}
| \mathcal{V}_{M_{q} \tau_{\theta} g} M_{\eta} \tau_{\mu} f(x, \omega)| \leq c \int_{\mathbb{R}^d} |K_{-}(t, \omega)| |\tau_{x} M_q \tau_{\theta} g(t)| |M_{\eta} \tau_{\mu} f(t)| \, dt.
\end{equation}
It is not possible to use the ordinary formula of translation for computing $ \tau_x M_q \tau_{\theta} g$, because we are not translating a radial function. Thus by formula \eqref{tra3} and the property \eqref{PF2} we obtain
\begin{eqnarray*}
\tau_{x} M_q \tau_{\theta} g(t) &=& (2 \pi)^{- \frac{d}{2}} \int_{\mathbb{R}^d} K_{-}(z,t)	K_{-}(x, z) \mathcal{F_{-}}(M_q \tau_{\theta} g)(z) \, dz\\
&=& (2 \pi)^{- \frac{d}{2}} \int_{\mathbb{R}^d} K_{-}(z,t)	K_{-}(x, z) \tau_{q} M_{\theta}  \mathcal{F_{-}}(g)(z) \, dz.
\end{eqnarray*}
Therefore
\begin{equation}
\label{R2}
\tau_{x} M_q \tau_{\theta} g(t)=(2 \pi)^{- \frac{d}{2}} \int_{\mathbb{R}^d} K_{-}(z,t)	K_{-}(x, z) \tau_{q} M_{\theta}  \mathcal{F_{-}}(g)(z) \, dz.
\end{equation}
Putting \eqref{ref5} and \eqref{R2} in \eqref{R0} we get
\begin{eqnarray*}
| \mathcal{V}_{M_{q} \tau_{\theta} g} M_{\eta} \tau_{\mu} f(x, \omega)| &\leq&  c \int_{\mathbb{R}^{3d}} |K_{-}(t, \omega)| |K_{-}(z,t)|	|K_{-}(x, z)|  |K_{-}(\eta, t)| |K_{-}(\xi, t)||K_{-}(\mu, \xi)| \cdot \\
&& \cdot  |\tau_{q} M_{\theta}  \mathcal{F_{-}}(g)(z)|  | \mathcal{F}_{-}f(\xi)| \, dz d \xi dt.
\end{eqnarray*}
Using the upper bound of the kernel (see Lemma \eqref{Ker1}) and the Fubini's theorem we obtain
\begin{eqnarray*}
| \mathcal{V}_{M_{q} \tau_{\theta} g} M_{\eta} \tau_{\mu} f(x, \omega)| &\leq&  c (1+|\underline{\omega}|)^{\lambda}(1+|\underline{x}|)^{\lambda}(1+|\underline{\eta}|)^{\lambda} (1+|\underline{\mu}|)^{\lambda} \int_{\mathbb{R^d}} (1+|\underline{t}|)^{4 \lambda} \, dt \cdot \\
&& \cdot \int_{\mathbb{R}^d} (1+|\underline{z}|)^{2 \lambda} |\tau_{q} M_{\theta}  \mathcal{F_{-}}(g)(z)| \, dz \int_{\mathbb{R}^d} (1+|\underline{\xi}|)^{2 \lambda} | \mathcal{F_{-}}(f)(\xi)| \, d \xi.
\end{eqnarray*}
Since $ \lambda=\frac{d-2}{2}$ and $d>2$ is even we have that the integral $ \int_{\mathbb{R^d}} (1+|\underline{t}|)^{4 \lambda} \, dt$ is not convergent. Therefore it is not possible to have an estimate for $| \mathcal{V}_{M_{q} \tau_{\theta} g} M_{\eta} \tau_{\mu} f(x, \omega)|$.

A similar reasoning proves that we cannot obtain estimates also for $| \mathcal{V}_{ \tau_{\theta}M_{q} g} M_{\eta} \tau_{\mu} f(x, \omega)|$, $| \mathcal{V}_{ \tau_{\theta}M_{q} g}  \tau_{\mu} M_{\eta}f(x, \omega)|$, $| \mathcal{V}_{ M_{q} \tau_{\theta} g} \tau_{\mu} M_{\eta}  f(x, \omega)|$.
\end{nb}
\begin{nb}
It is not possible to repeat the same estimates of Proposition \ref{A0} and Proposition \ref{A2} using as a window function the modulation of $g$. Below, we perform the computation to show where the problem arises. Let $ \theta, \mu, \eta, x, \omega \in \mathbb{R}^d$, we want to make an estimate of $|\mathcal{V}_{M_{\theta} g} \tau_{\mu} M_{\eta} f(x, \omega)|$. From the Definition \ref{NDE} we have
\begin{equation}
\label{R7}
|\mathcal{V}_{M_{\theta} g} \tau_{\mu} M_{\eta} f(x, \omega)| \leq c \int_{\mathbb{R}^d} |K_{-}(t, \omega)| |\tau_x M_{\theta} g(t)| |\tau_{\mu} M_{\eta} f(t)| \, dt.
\end{equation}
Since $M_{\theta} g$ is no longer radial we have to use the formula \eqref{tra3} for computing the translation. Thus by formula \eqref{eq1} we get
\begin{eqnarray*}
\tau_{x}M_{\theta} g(t) &=& (2 \pi)^{- \frac{d}{2}} \int_{\mathbb{R}^d}K_{-}(z,t)K_{-}(x, z) \mathcal{F_{-}}(M_{\theta} g)(z) \, dz\\
&=& (2 \pi)^{- \frac{d}{2}} \int_{\mathbb{R}^d}K_{-}(z,t)K_{-}(x, z) \tau_{\theta} \mathcal{F_{-}}( g)(z) \, dz.
\end{eqnarray*}
Therefore
\begin{equation}
\label{R8}
\tau_{x}M_{\theta} g(t)=(2 \pi)^{- \frac{d}{2}} \int_{\mathbb{R}^d}K_{-}(z,t)K_{-}(x, z) \tau_{\theta} \mathcal{F_{-}}( g)(z) \, dz.
\end{equation}
Putting \eqref{R8} and \eqref{ref2} in \eqref{R7} and using the upper bound of the kernel we obtain
\begin{eqnarray*}
|\mathcal{V}_{M_{\theta} g} \tau_{\mu} M_{\eta} f(x, \omega)| &\leq& c \int_{\mathbb{R}^d} |K_{-}(t, \omega)| |K_{-}(z,t)| |K_{-}(x, z)| |\tau_{\theta} \mathcal{F_{-}}( g)(z)|\cdot \\  &&\cdot|K_{-}(\xi,t)||K_{-}(\mu, \xi)| | \tau_{\eta} \mathcal{F_{-}}(f)(\xi)| \, dz d \xi dt\\
&\leq& c(1+|\underline{\omega}|)^{\lambda}(1+|\underline{x}|)^{\lambda}(1+|\underline{\mu}|)^{\lambda} \int_{\mathbb{R}^d} (1+|\underline{t}|)^{3 \lambda} dt \cdot \\
&& \cdot \int_{\mathbb{R}^d}(1+|\underline{z}|)^{2 \lambda} |\tau_{\theta} \mathcal{F_{-}}(g)(z)| dz\int_{\mathbb{R}^d}(1+|\underline{\xi}|)^{2 \lambda} |\tau_{\eta} \mathcal{F_{-}}(f)(\xi)| d\xi.
\end{eqnarray*}
As in the previous remark the integral $\int_{\mathbb{R}^d} (1+|\underline{t}|)^{3 \lambda} dt$ does not converge, so it not possible to have an estimate. The same considerations hold also for $|\mathcal{V}_{M_{\theta} g}  M_{\eta}\tau_{\mu} f(x, \omega)|$.
\end{nb}

\section{Further properties of the Clifford short-time Fourier transform }
There are some properties for the Clifford short-time Fourier transform which hold also for the complex case and quaternionic case (see \cite{BA}), but with some differences.

\begin{theorem}[Orthogonality relation]
\label{ort1}
Let $d>2$ and even. Let $g_1,g_2 \in \mathcal{S}(\mathbb{R}^d) $ be radial functions and $ f_1, f_2 \in L^{2}(\mathbb{R}^d) \otimes \mathbb{R}_d$. Then we have
\begin{equation}
\label{Ort}
\int_{\mathbb{R}^{2d}} \overline{\mathcal{V}_{g_1}f_1(x, \omega)} \mathcal{V}_{g_2}f_2(x, \omega) \, d \omega \, dx = \langle f_1, f_2 \rangle \biggl( \int_{\mathbb{R}^d} g_1(z) g_2(z) \, dz \biggl).
\end{equation}
\end{theorem}
\begin{proof}
From Definition \ref{NDE} and the Plancherel's theorem (see Proposition \ref{Plan}) we have
\begin{eqnarray*}
\int_{\mathbb{R}^{2d}} \overline{\mathcal{V}_{g_1}f_1(x, \omega)} \mathcal{V}_{g_2}f_2(x, \omega) \, d \omega \, dx  &=& \int_{\mathbb{R}^{2d}} \overline{\mathcal{F_{-}}(\tau_x \bar{g}_1 \cdot f_1)(\omega)}\mathcal{F_{-}}(\tau_x \bar{g}_2 \cdot f_2)(\omega) \, d \omega \, dx\\
&=& \int_{\mathbb{R}^{2d}} \overline{(\tau_x \bar{g}_1 \cdot f_1)}(t) (\tau_x \bar{g}_2 \cdot f_2)(t) \, dt \, dx.
\end{eqnarray*}
Now, since $g_1$ and $g_2$ are real valued functions we can omit the conjugate. Moreover, the window functions $g_1$, $g_2$ can commute. Thus
$$\int_{\mathbb{R}^{2d}} \overline{\mathcal{V}_{g_1}f_1(x, \omega)} \mathcal{V}_{g_2}f_2(x, \omega) \, d \omega \, dx = \int_{\mathbb{R}^{2d}} \overline{f_1}(t) f_2(t) g_1(t-x) g_2(t-x) \, dt \, dx.$$
By hypothesis we can use Fubini's theorem for changing the order of integration
$$ \int_{\mathbb{R}^{2d}} \overline{\mathcal{V}_{g_1}f_1(x, \omega)} \mathcal{V}_{g_2}f_2(x, \omega)\, d \omega \, dx  = \int_{\mathbb{R}^{d}} \overline{f_1}(t) f_2(t) \biggl( \int_{\mathbb{R}^d} g_1(t-x) g_2(t-x) \, dx \biggl)\, dt .$$
Finally, by a change of variable we get
\begin{eqnarray*}
\int_{\mathbb{R}^{2d}} \overline{\mathcal{V}_{g_1}f_1(x, \omega)} \mathcal{V}_{g_2}f_2(x, \omega) \, d \omega \, dx 
&=&\int_{\mathbb{R}^d} \overline{f_1}(t) f_2(t)  \biggl( \int_{\mathbb{R}^d} g_1(z) g_2(z) \, dz \biggl) \, dt\\
&=& \langle f_1, f_2 \rangle \biggl( \int_{\mathbb{R}^d} g_1(z) g_2(z) \, dz \biggl).
\end{eqnarray*}
\end{proof}
\begin{nb}
We can prove the above theorem also using Lemma \ref{tensor}, and this proof may be of interest in some other contexts. Supposing the same hypothesis of Theorem \ref{ort1}, by Plancherel's theorem  and the fact that $g_1$ and $g_2$ are real valued we have
\begin{eqnarray*}
\int_{\mathbb{R}^{2d}} \overline{\mathcal{V}_{g_1}f_1(x, \omega)} \mathcal{V}_{g_2}f_2(x, \omega) \, d \omega \, dx  &=& \int_{\mathbb{R}^{2d}} \overline{\mathcal{F}_{2^{-}} \mathcal{T}(f_1 \otimes g_1)(x, \omega)} \mathcal{F}_{2^{-}} \mathcal{T}(f_2 \otimes g_2)(x, \omega) \, d \omega \, dx\\
&=& \int_{\mathbb{R}^{2d}}  \overline{\mathcal{T}(f_1 \otimes g_1)(x, t)} \mathcal{T}(f_2 \otimes g_2)(x, t) \, d t \, dx\\
&=& \int_{\mathbb{R}^{2d}} \overline{f_1}(t) f_2(t) g_1(t-x) g_2(t-x) \, dt \, dx.
\end{eqnarray*}
Using the same arguments of Theorem \ref{ort1} we obtain the equality \eqref{Ort}.
\end{nb}

\begin{Cor}
Let $d>2$ and even. If $ g \in \mathcal{S}(\mathbb{R}^d)$ is a radial function and $ f \in L^2(\mathbb{R}^d) \otimes \mathbb{R}_d$ then
$$ \| \mathcal{V}_{g} f(x, \omega) \|_{2}^2= \| f \|_2^2  \, \, \int_{\mathbb{R}^d} g^2(z) \, dz.$$
\end{Cor}
\begin{proof}
If we put $f_1=f_2:=f$ and $g_1=g_2:=g$ in the equality \eqref{Ort} we obtain the thesis.
\end{proof}
\begin{theorem}[Reconstruction formula]
\label{Rcn1}
Let us assume that $ f \in L^{2}(\mathbb{R}^d) \otimes \mathbb{R}_d$ , $g \in \mathcal{S}(\mathbb{R}^d)$ is a radial function and $ \int_{\mathbb{R}^d} g^2(z) \, dz \neq 0$. Then for all $ f \in L^{2}(\mathbb{R}^d) \otimes \mathbb{R}_d$ we have
$$ f(y)= \frac{1}{\int_{\mathbb{R}^d} g^2(z) \, dz} \int_{\mathbb{R}^{2d}}M_{\omega} \tau_x g(y) \mathcal{V}_{g}f(x, \omega) \, d\omega \, dx.$$
\end{theorem}
\begin{proof}
Let us assume
$$ \widetilde{f}(y):= \frac{1}{\int_{\mathbb{R}^d} g^2(z) \, dz} \int_{\mathbb{R}^{2d}}M_{\omega} \tau_x g(y) \mathcal{V}_{g}f(x, \omega) \, d\omega \, dx.$$
Let $ h \in L^{2}(\mathbb{R}^d) \otimes \mathbb{R}_d$. By the equality \eqref{imp}, the orthogonality relation (see Theorem \ref{ort1}) and Fubini's theorem we obtain
\begin{eqnarray*}
\langle \tilde{f}, h \rangle &=& \int_{\mathbb{R}^d} \overline{\widetilde{f}(y)} h(y) \, dy= \frac{1}{\int_{\mathbb{R}^d} g^2(z) \, dz} \biggl( \int_{\mathbb{R}^{3d}} \overline{M_{\omega} \tau_x g(y) \mathcal{V}_{g}f(x, \omega)} \, d\omega \, dx \biggl) h(y) \, dy\\
&=& \frac{1}{\int_{\mathbb{R}^d} g^2(z) \, dz} \int_{\mathbb{R}^{3d}} \overline{M_{\omega} \tau_x g(y) \mathcal{V}_{g}f(x, \omega)} h(y) \, d\omega \, dx  \, dy \\
&=& \frac{1}{\int_{\mathbb{R}^d} g^2(z) \, dz} \int_{\mathbb{R}^{3d}}\overline{\mathcal{V}_{g}f(x, \omega)} \, \, \overline{M_{\omega} \tau_x g(y)} \,  h(y) \, dy \, d\omega \, dx   \\
&=& \frac{1}{\int_{\mathbb{R}^d} g^2(z) \, dz} \int_{\mathbb{R}^{2d}}\overline{\mathcal{V}_{g}f(x, \omega)} \biggl( \int_{\mathbb{R}^d} \overline{M_{\omega} \tau_x g(y)} \,  h(y) \, dy \biggl) \, d\omega \, dx   \\
&=& \frac{1}{\int_{\mathbb{R}^d} g^2(z) \, dz} \int_{\mathbb{R}^{2d}}\overline{\mathcal{V}_{g}f(x, \omega)} \mathcal{V}_{g}h(x, \omega) \, d \omega \, dx\\
&=& \frac{1}{\int_{\mathbb{R}^d} g^2(z) \, dz} \langle f, h \rangle \int_{\mathbb{R}^d} g^2(z) \, dz= \langle f, h \rangle.
\end{eqnarray*}
We proved that for all $h \in L^{2}(\mathbb{R}^d) \otimes \mathbb{R}_d$
$$ \langle \widetilde{f}, h \rangle= \langle f, h \rangle.$$
This implies $ \widetilde{f}(y)=f(y).$
\end{proof}
The inversion formula gives us the possibility to write the Clifford short-time Fourier transform using the reproducing kernel associated to the Clifford Gabor space, introduced in \cite{AF}, defined by 

$$\mathcal{G}^{g}_{\mathbb{R}_d}:=\{\mathcal{V}_g f, \quad  f \in L^2(\mathbb{R}^d) \otimes \mathbb{R}_d\},$$
where $g$ is a radial window function.
\begin{theorem}[Reproducing kernel]
Let $d>2$ and even. If $ f \in L^{2}(\mathbb{R}^d) \otimes \mathbb{R}_d$ and $g$ is a radial function in $ \mathcal{S}(\mathbb{R}^d)$, we have that
 $$ \mathbb{K}_g(\omega,x ; \omega', x')= \frac{1}{(2 \pi)^{\frac{d}{2}}\int_{\mathbb{R}^d} g^2(z) \, dz} \int_{\mathbb{R}^d} K_{-}(t, \omega') g(t-x') \overline{K_{-}(t, \omega) g(t-x)} \, dt$$
is a reproducing kernel, i.e
$$ V_g f(x', \omega')=\int_{\mathbb{R}^{2d}}  \mathbb{K}_g(\omega,x ; \omega', x') V_g f(x, \omega) \, d \omega \, dx .$$
\end{theorem}
\begin{proof}
By the reconstruction formula (see Theorem \ref{Rcn1}), Fubini's theorem and since $g$ is a real valued function we have
\begin{eqnarray*}
\mathcal{V}_g f(x', \omega') &=& (2 \pi)^{- \frac{d}{2}} \int_{\mathbb{R}^d} K_{-}(t,\omega') g(t-x') f(t) \, dt\\
&=& (2 \pi)^{- \frac{d}{2}} \int_{\mathbb{R}^d} K_{-}(t,\omega') g(t-x') \frac{1}{\int_{\mathbb{R}^d} g^2(z) \, dz} \biggl( \int_{\mathbb{R}^{2d}}M_{\omega} \tau_x g(t) \mathcal{V}_{g}f(x, \omega) \, d\omega \, dx \biggl) \, dt \\
&=& (2 \pi)^{- \frac{d}{2}} \int_{\mathbb{R}^{3d}} K_{-}(t,\omega') g(t-x') \frac{1}{\int_{\mathbb{R}^d} g^2(z) \, dz} M_{\omega} \tau_x g(t) \mathcal{V}_{g}f(x, \omega)  d\omega  dx   dt \\
&=& (2 \pi)^{- \frac{d}{2}} \int_{\mathbb{R}^{3d}} K_{-}(t,\omega') g(t-x') \frac{1}{\int_{\mathbb{R}^d} g^2(z) \, dz} K_{-}(\omega, t) g(t-x) \mathcal{V}_{g}f(x, \omega) \, dt \, d\omega \, dx   \\
&=&  \! \! \! \int_{\mathbb{R}^{2d}} \!  \biggl( \int_{\mathbb{R}^d} \! K_{-}(t,\omega') g(t-x') \frac{1}{(2 \pi)^{\frac{d}{2}} \int_{ \mathbb{R}^d} g^2(z) \, dz} K_{-}(\omega, t) g(t-x)   dt \biggl) \mathcal{V}_{g}f(x, \omega) d\omega  dx   \\
&=& \int_{\mathbb{R}^{2d}}  \mathbb{K}_g(\omega,x ; \omega', x') \mathcal{V}_g f(x, \omega) \, d \omega \, dx.
\end{eqnarray*}
\end{proof}
For this reproducing kernel it is possible to have the following bound.
\begin{Cor}
Let $d>2$ and even. Let us assume that $ f \in L^{2}(\mathbb{R}^d) \otimes \mathbb{R}_d$ and $g \in \mathcal{S}(\mathbb{R}^d)$ is a radial function, then we have that
$$ |\mathbb{K}_g(\omega,x ; \omega', x')| \leq C (1+|\underline{\omega}|)^{\lambda} (1+|\underline{\omega'}|)^{\lambda} (1+|\underline{x}|)^{2 \lambda} (1+|\underline{x'}|)^{ 2 \lambda},$$
where $C$ is a constant.
\end{Cor}
\begin{proof}
By the relation \eqref{inv} and the fact that $g$ is a real valued function we have that
$$ \mathbb{K}_g(\omega,x ; \omega', x')= \frac{1}{(2 \pi)^{\frac{d}{2}}\int_{\mathbb{R}^d} g^2(z) \, dz} \int_{\mathbb{R}^d} K_{-}(t, \omega')  K_{-}(\omega, t) g(t-x')g(t-x) \, dt.$$
Using the upper bound of the kernel $K_{-}$ (see Lemma \ref{Ker1}) we get
\begin{eqnarray*}
 |\mathbb{K}_g(\omega,x ; \omega', x')| &\leq&  \frac{c}{(2 \pi)^{\frac{d}{2}} \left |\int_{\mathbb{R}^d} g^2(z) \, dz \right|} \int_{\mathbb{R}^d} |K_{-}(t, \omega') || K_{-}(\omega, t)| |g(t-x')| |g(t-x)| \, dt\\
 & \leq & \frac{c}{(2 \pi)^{\frac{d}{2}}\left| \int_{\mathbb{R}^d} g^2(z) \, dz \right|} (1+|\underline{\omega}|)^{\lambda} (1+|\underline{\omega'}|)^{\lambda} \int_{\mathbb{R}^d} (1+|\underline{t}|)^{2 \lambda}  |g(t-x')| |g(t-x)| \, dt,
\end{eqnarray*}
where $c$ is a positive constant. If we put $t=z+x+x'$ we get
\begin{eqnarray*}
|\mathbb{K}_g(\omega,x ; \omega', x')| &\leq&  \frac{c}{(2 \pi)^{\frac{d}{2}} \left|\int_{\mathbb{R}^d} g^2(z) \, dz \right|} (1+|\underline{\omega}|)^{\lambda} (1+|\underline{\omega'}|)^{\lambda}\cdot \\ 
&& \cdot \int_{\mathbb{R}^d} (1+|\underline{z}|+|\underline{x}|+ |\underline{x'}|)^{2 \lambda} |g(z+x)| |g(z+x')| \, dz\\
& \leq & \frac{c}{(2 \pi)^{\frac{d}{2}}\left |\int_{\mathbb{R}^d} g^2(z) \, dz \right|} (1+|\underline{\omega}|)^{\lambda} (1+|\underline{\omega'}|)^{\lambda} (1+|\underline{x}|)^{2 \lambda} (1+|\underline{x'}|)^{ 2 \lambda} \cdot \\
&& \cdot \int_{\mathbb{R}^d} (1+|\underline{z}|)^{2 \lambda} |g(z+x)| |g(z+x')| \, dz .
\end{eqnarray*}

We denote as $C:=\frac{ c \int_{\mathbb{R}^d} (1+|\underline{z}|)^{2 \lambda} |g(z+x)| |g(z+x')| \, dz}{(2 \pi)^{\frac{d}{2}} \left| \int_{\mathbb{R}^d} g^2(z) \, dz \right|}$. Thus we get
$$ |\mathbb{K}_g(\omega,x ; \omega', x')| \leq C (1+|\underline{\omega}|)^{\lambda} (1+|\underline{\omega'}|)^{\lambda} (1+|\underline{x}|)^{2 \lambda} (1+|\underline{x'}|)^{ 2 \lambda}.$$
\end{proof}

\section{Lieb's Uncertainty Principle}
In this section we want to extend the Lieb's Uncertainty Principle \cite{L} in the Clifford setting. Firstly let us recall that in general the uncertainty principles state that a function and its Fourier transform cannot be simultaneously sharply localized. This is emphasised by the following generic principle (see \cite{G})

\begin{center}
\emph{A function cannot be concentrated on small sets in the time-frequency plane, no matter which time-frequency representation is used}.
\end{center}

Before to prove a weak uncertainty principle for the Clifford short-time Fourier transform we go through the following important estimate.
\begin{prop}
\label{Tfs}
Let $d>2$ and even.
If $g$ is a radial function in $W_{ p \lambda}(\mathbb{R}^d)$ , with $ p \geq 1$, then
$$ \| M_{\omega} \tau_x g \|_p \leq c(1+|\underline{\omega}|)^{\lambda} (1+|\underline{x}|)^{\lambda} \| g\|_{W_{ p \lambda}},$$
where $c$ is a positive constant.
\end{prop}
\begin{proof}
Since $g$ is a radial function by Proposition \ref{ratra} we get
$$ \| M_{\omega} \tau_x g \|_p^p= \int_{\mathbb{R}^d} |K_{-}(\omega,t) g(t-x)|^p \, dt.$$
Now, if we put $ s=t-x$ by the upper bound of Lemma \ref{Ker1} we obtain

\begin{eqnarray*}
 \| M_{\omega} \tau_x g \|_p^p &\leq & c \int_{\mathbb{R}^d} |K_{-}(\omega, s+x)|^p |g(s)|^p \, ds \\
 &\leq &  c \int_{\mathbb{R}^d} (1+|\underline{\omega}|)^{\lambda p} (1+ |\underline{s}+ \underline{x}|)^{\lambda p} |g(s)|^p \, ds\\
 &\leq&  c (1+|\underline{\omega}|)^{\lambda p} (1+ |\underline{x}|)^{\lambda p} \int_{\mathbb{R}^d}  (1+ |\underline{s}|)^{\lambda p} |g(s)|^p \, ds\\
  &=& c (1+|\underline{\omega}|)^{\lambda p} (1+ |\underline{x}|)^{\lambda p}  \|g \|_{W_{ p \lambda}}^p.
\end{eqnarray*}
\end{proof}
\begin{nb}
This result it is very different from the classical result of the Fourier analysis, in which we have $  \| M_{\omega} \tau_x f \|_p = \| f\|_{p}$.
\end{nb}

\begin{prop}[Weak uncertainty principle]
\label{L1}
Let $d>2$ and even.
Suppose that $ \| f \|_2= \| g \|_{W_{2 \lambda }}=1$, with $g$ a radial function. Let $U \subseteq \mathbb{R}^{2d}$ be an open set and $ \varepsilon \geq 0$ such that
\begin{equation}
\label{new9}
\int \int_{U}(1+| \underline{x}|)^{-2 \lambda} (1+| \underline{\omega}|)^{-2 \lambda} |\mathcal{V}_g f(x, \omega)|^2 \, d x \, d \omega \geq 1- \varepsilon.
\end{equation}
Then $|U| \geq (1- \varepsilon) \frac{1}{c},$ where $c$ is positive constant different from zero.
\end{prop}
\begin{proof}
Formula \eqref{imp}, H\"{o}lder inequality and Proposition \ref{Tfs} imply that
\begin{eqnarray*}
|\mathcal{V}_{g}f(x, \omega)| &=&  \int_{\mathbb{R}^d}|\overline{M_{\omega} \tau_x g(t)}| |f(t)| \, dt \leq  \| M_{\omega} \tau_x g \|_2 \| f \|_2 \\
&\leq & c(1+| \underline{x}|)^{\lambda} (1+| \underline{\omega}|)^{\lambda} \| g \|_{W_{2 \lambda }} \| f \|_2 =  c (1+| \underline{x}|)^{\lambda} (1+| \underline{\omega}|)^{\lambda}.
\end{eqnarray*}
Therefore, using this calculus we get
\begin{eqnarray*}
1- \varepsilon & \leq& \int \int_{U}(1+| \underline{x}|)^{-2 \lambda} (1+| \underline{\omega}|)^{-2 \lambda} |\mathcal{V}_g f(x, \omega)|^2 \, d x \, d \omega  \\
& \leq & c \int \int_{U}(1+| \underline{x}|)^{-2 \lambda} (1+| \underline{\omega}|)^{-2 \lambda} (1+| \underline{x}|)^{2 \lambda} (1+| \underline{\omega}|)^{ 2\lambda} \, d x \, d \omega\\
&=& c|U|.
\end{eqnarray*}
hence $|U| \geq (1- \varepsilon) \frac{1}{c}$.
\end{proof}

\begin{nb}
The Lieb's uncertainty principle in our context is quite different from the classical one, since we have the presence of polynomials in \eqref{new9}, but these are crucial for the convergence of the integrals.
\end{nb}

\begin{nb}
In order to improve the above proposition, in the classical Fourier analysis, Lieb \cite{L} estimated the $L^p$- norm of the short-time Fourier transform using the Hausdorff-Young's inequality. However, as remarked in \cite{LLF}, it is not possible to extend this inequality in the Clifford setting since we cannot use the Riesz interpolation theorem.
\end{nb}

\textbf{Acknowledgements}
\newline
The author is grateful to the anonymous referee whose deep and extensive
comments greatly contributed to improve this paper. I would like to thank Prof.~Irene Sabadini for reading an earlier version of this paper and for her interesting comments.

\hspace{3mm}

\noindent
Antonino De Martino,
Dipartimento di Matematica \\ Politecnico di Milano\\
Via Bonardi n.~9\\
20133 Milano\\
Italy

\noindent
\emph{email address}: antonino.demartino@polimi.it\\
\emph{ORCID iD}: 0000-0002-8939-4389


\begin{thebibliography}{AAAA}
\bibitem[1]{AF} L.D Abreu, H.G Feichtinger, \emph{Function Spaces of Polyanalytic Functions}, Harmonic and Complex Analysis and Its applications, 1–38, Trends Math., Birkhäuser/Springer, Cham, (2014). 
\bibitem[2]{BA} M.~Bahri, R.Ashino, \emph{Two-dimensional quaternionic wiondow Fourier Transform}, in Fourier Transforms - Approach to Scientific Principles, InTechOpen (G.S. Nikolic), London, 2011.
\bibitem[3]{BDS1} F.~Brackx, R.~Delanghe, F.~Sommen, \emph{Clifford analysis}, Reserch Notes in Mathematics, Vol. 76, Pitman (Advanced Publishing Program), Boston, MA, 1982.
\bibitem[4]{BDS} R.~Bujack, H.~De Bie, N.~De Schepper, \emph{Convolution products for Hypercomplex Fourier Transforms}, J. Math Imaging Vis. (\textbf{48}) 3 (2014), 606-624.
\bibitem[5]{CHO} P.~Cerejeiras, S.~Hartmann, H.~Orelma, \emph{Structural results for quaternionic Gabor frames}, Adv. Appl. Clifford Algebr. (\textbf{28}) 5 (2018), 1-12.
\bibitem[6]{CK} P.~Cerejeiras, U.~Kähler, \emph{Monogenic Signal Theory}, In D.~Alpay (ed.) Operator Theory, Springer, Basel, (2014).
\bibitem[7]{CDL} D.~Constales, H.~De Bie, P.~Lian, \emph{A new construction of the Clifford-Fourier kernel}, J. Fourier Anal. Appl. (\textbf{23}) (2017), 462-483.
\bibitem[8]{CSSS} F.~Colombo, I.~Sabadini, I.~Sommen, D.C.~Struppa, \emph{Analysis of Dirac Systems and Computational Algebra}, Vol. 39, Birk\"{a}user, Boston, 2004.
\bibitem[9]{FKC} Y.~Fu, U.~Kähler, P.~Cerejeiras, \emph{The Balian-Low theorem for the windowed quaternionic Fourier transform}, Adv. Appl. Clifford Algebr. (\textbf{22}) 4 (2012), 1025–1040.
\bibitem[10]{D} H.~De Bie, \emph{Fourier Transforms in Clifford analysis}, In D.~Alpay (ed.) Operator Theory, Springer, Basel, (2014).
\bibitem[11]{DS} H.~De Bie, N.~De Schepper, \emph{The fractional Clifford-Fourier Transform}, Complex Anal. Oper. Theory (\textbf{6}) 5 (2012), 1047-1067.
\bibitem[12]{DSS} H.~De Bie, N.~De Schepper, F.~Sommen, \emph{The class of Clifford-Fourier transforms,} J.Fourier Anal. Appl. (\textbf{17}) 6 (2011), 1198-1231.
\bibitem[13]{DOV} H.~De Bie, R.~Oste, J.~Van der Jeugt, \emph{Generalized Fourier transforms
arising from the enveloping algebras of sl(2) and $osp(1 \mid 2)$}, Int. Math. Res. Not. IMRN (\textbf{15}) (2016), 4649-4705.
\bibitem[16]{DX}H.~De Bie, Y.~Xu, \emph{On the Clifford-Fourier transform}, Int Math Res Not IMRN (\textbf{22}) (2011), 5123-5163.
\bibitem[17]{DSS1} R.~Delanghe, F.~Sommen, V.~Soucek, \emph{Clifford Algebra and Spinor-Valued Functions}, Mathematics and its Applications, Vol.53, Dordrecht: Kluwer Academic Publishers Group, 1992.
\bibitem[18]{DK} A.~De Martino, K.~Diki, \emph{On the quaternionic short-time Fourier and Segal-Bargamann transforms},  Mediterr. J. Math. (\textbf{110}) 18 (2021).
\bibitem[19]{EJ} J.~El Kamel, R.~Jday, \emph{Uncertainty Principles for the Clifford-Fourier Transform}, Adv. Appl. Clifford Algebras (\textbf{27}) 3 (2017), 2429-2443.
\bibitem[20]{EJ1} J.~El Kamel, R.~Jday, \emph{Inequalities in the setting of Clifford Analysis}, Math Phys Anal Geom (\textbf{21}) 4 (2018), 1-18 .
\bibitem[21]{GM} J.~Gilbert, M.~Murray, \emph{Clifford Algebras and Dirac Operators in Harmonic Analysis}, Vol. 26, Cambridge University Press, Cambridge, 1991.
\bibitem[22]{G} K.~Gr\"{o}chenig, \emph{Foundations of Time-Frequency Analysis}, Birkh\"{a}user, Boston, 2001.
\bibitem[23]{LLF} S.~Li, J.~Leng, M.~Fei, \emph{Paley-Wiener-Type theorems for the Clifford-Fourier transform}, Math. Meth. Appl. Sci. (\textbf{42}) 18 (2019), 6101-6113.
\bibitem[24]{L} E.~H.Lieb, \emph{Integral bounds for radar ambiguity functions and Wigner distribution}, J. Math. Phys. (\textbf{31}) 3 (1990), 594-599.
\bibitem[25]{S} F.~Sommen, \emph{Special functions in Clifford Analysis and axial symmetry}, J. Math. Anal. Appl. (\textbf{130}), 1 (1988), 110-133.























\end{thebibliography}
\end{document}